\documentclass[a4paper,12pt,leqno,english]{article}
\usepackage[T1]{fontenc}
\usepackage[utf8]{inputenc}
\usepackage[a4paper, margin=2.5cm]{geometry}
\usepackage{babel}
\usepackage{amsthm,amsmath,amsfonts,amssymb,titlesec}
\usepackage{enumitem,pifont,physics,comment}
\usepackage{hyperref}
\newtheorem{theorem}{Theorem}[section]
\newtheorem{proposition}{Proposition}[section]
\newtheorem{lemma}{Lemma}[section]
\makeatletter
\renewcommand\theequation%
{\thesection.\arabic{equation}}
\@addtoreset{equation}{section}
\makeatother
\usepackage{newtxtext,newtxmath}

\usepackage[
backend=biber,
style=alphabetic,
sorting=nyt
]{biblatex}
\usepackage{csquotes}
\usepackage{amsmath,amsthm,amssymb,amscd,color, xcolor,mathtools,url,tikz}
\usepackage{tikz}
\usetikzlibrary{positioning}

\def\R{\mathbb{R}}
\def\N{\mathbb{N}}
\def\Z{\mathbb{Z}}

\def\C{\mathbb{C}}

\let\eps\varepsilon
\let\epsilon\varepsilon

\newcommand{\scal}[2]{\langle #1, #2 \rangle}

\let\oldsum\sum
\renewcommand{\sum}{\displaystyle\oldsum}
\let\oldprod\prod
\renewcommand{\prod}{\displaystyle\oldprod}
\let\oldinf\inf
\renewcommand{\inf}{\displaystyle\oldinf}
\let\oldsup\sup
\renewcommand{\sup}{\displaystyle\oldsup}

\let\leq\leqslant
\let\geq\geqslant
\newlength{\oldparindent}
\newcommand{\myindent}{\hspace{\oldparindent}}

\providecommand{\keywords}[1]
{
  \small	
  \textbf{Keywords---} #1
}

\title{\textbf{A Weakly Turbulent solution to the cubic Nonlinear Harmonic Oscillator on $\R^2$ perturbed by a real smooth potential decaying to zero at infinity}}
\author{Maxine Chabert}

\begin{document}
\maketitle

\begin{abstract}
     We build a smooth real potential $V(t,x)$ on $(t_0,+\infty)\times \R^2$ decaying to zero as $t\to \infty$ and a smooth solution to the associated perturbed cubic Nonlinear Harmonic Oscillator whose Sobolev norms blow up logarithmically as $t\to \infty$. Adapting the method of Faou and Raphael for the linear case, we modulate the solitons associated to the Nonlinear Harmonic Oscillator by time-dependent parameters solving a quasi-Hamiltonian dynamical system whose action grows up logarithmically, thus yielding logarithmic growth for the Sobolev norm of the solution.
\end{abstract}


\section{Introduction}
\subsection{\textit{Setting of the problem and main result}} 
\myindent We introduce the linear operator acting on $L^2(\R^2;\C)$ associated to the quantum Harmonic Oscillator on $\R^2$

\begin{equation}
    H := -\Delta + |x|^2
\end{equation}

where for $x = (x_1,x_2) \in \R^2$ we denote $|x|^2 = x_1^2 + x_2^2$ and $\Delta$ is the Laplace operator. To this operator are associated the modified Sobolev function spaces defining the domain of $H$ 

\begin{equation}
    \forall r\geq 0 \qquad H_x^r  := \left\{u\in L^2(\R^2;\C) \quad H^{\frac{r}{2}} u\in L^2\right\}
\end{equation}

\myindent An open problem is to exhibit a solution to the nonlinear defocusing Schrödinger-type equation

\begin{equation}\label{unperteq}
    i\partial_t u = Hu + u|u|^2
\end{equation}

whose $H_x^r$ norm blows up in infinite time i.e. such that for some $r > 0$

\begin{equation}
    \lim_{t\to \infty} \|u(t)\|_{H_x^r} = +\infty
\end{equation}

\myindent As the $L^2$ norm is preserved by equation \eqref{unperteq}, such a phenomenon would be the result of energy transfer from low to high frequencies generating growth of Sobolev norms, thanks to the nonlinearity (indeed, should we remove the $u|u|^2$ term, the Sobolev norms would be preserved).

\myindent In this paper, we make a step toward this direction by building a smooth solution to the similar equation

\begin{equation}\label{eqintro}
    i\partial_t u = Hu + u|u|^2 + V u
\end{equation}

where $V(t,x)$ is a smooth real potential such that $V$ decays to zero when $t\to \infty$ in $H_x^r$ norms, that is

\begin{equation}
    \forall r\geq 0 \qquad \lim_{t\to \infty} \|V(t,\cdot)\|_{H_x^r}= 0
\end{equation}

\myindent We are able to exhibit a smooth solution $u$ of \eqref{eqintro} whose $H_x^1$ norm has infinite limit when $t\to \infty$. Precisely, we prove the following theorem

\begin{theorem}\label{main}
    There are a $t_0 >0$, a potential $V(t,x) \in \mathcal{C}^{\infty}((t_0, + \infty)\times \R^2 ; \R)$ and a function $u \in \mathcal{C}^{\infty}((t_0, +\infty)\times \R^2 ; \C)$ such that
    
    \begin{equation}\label{maineq}
        \forall (t,x) \in (t_0,+\infty)\times \R^2 \qquad i\partial_t u(t,x) = Hu(t,x) + u(t,x)|u(t,x)|^2 + V(t,x) u(t,x) 
    \end{equation}
    
    \myindent Moreover, for all $r\geq 0$ and $k\in \N$ 
    
    \begin{equation}
        \lim_{t\to +\infty} \|\partial_t^k V(t,x)\|_{H_x^r} = 0
    \end{equation}
    
    \myindent Finally, there are constants $c,C > 0$ such that
    
    \begin{equation}
        c\left(\log t\right)^{\frac{1}{2}} \leq \|u(t)\|_{H_x^1} \leq C \left(\log t\right)^{\frac{1}{2}} \qquad t\to \infty
    \end{equation}
    
\end{theorem}

\subsection{\textit{Earlier works}}

\myindent Firstly, the question of finding solution to \eqref{unperteq} with unbounded $H_x^r$ norm as $t\to \infty$ is very similar to the question of Bourgain \cite{bourgain2010problems} of finding solutions to the cubic defocusing NLS on the torus $\R^2 / \Z^2$ with unbounded $H^s$ norm. Indeed, the cubic defocusing NLS on the torus is the equation

\begin{equation}\label{eqbourg}
    i\partial_t u = -\Delta u + u|u|^2
\end{equation}

\myindent Now, it is reasonable to compare the Laplacian operator on the torus to the Harmonic Oscillator on the Euclidian space $\R^2$ as they both have a discreet spectrum formed of a sequence of eigenvalues. Thus, the question of transfer of energy from low frequencies to high frequencies in NLS is very similar to the same question for the Nonlinear Harmonic Oscillator on $\R^2$ \eqref{unperteq}. 

\myindent Important progress has been made which seem to indicate a positive answer to Bourgain's problem. Kuskin considered in \cite{kuksin1997oscillations} the similar equation

\begin{equation}
    -i\partial_t u = -\delta \Delta u + u|u|^2 \ \delta << 1
\end{equation}

\myindent He obtained solutions whose Sobolev norms grow by an inverse power of $\delta$ for all $\delta > 0$. See also related papers \cite{kuksin1995squeezing} \cite{kuksin1996growth} \cite{kuksin1997turbulence} \cite{kuksin1999spectral}.

\myindent The seminal paper \cite{colliander2010transfer} obtained for equation \eqref{eqbourg} growth of $H^s$ norm for $s> 1$ by an arbitrary factor in finite time for small initial data. Optimised bounds on the time of growth were obtained later in \cite{guardia2015growth}. Finally, \cite{hani2015modified} used this approach to exhibit \textit{infinite cascade solutions} that is global solutions with small initial date and unbounded $H^s$ norm as $t\to\infty$ for $s$ quite large.

\myindent Regarding the possible \textit{growth rate} of unbounded solutions, Bourgain conjectured moreover that in the case an unbounded solution exists on the torus, the growth of Sobolev norm should be subpolynomial that is

\begin{equation}
    \forall \eps > 0 \ \|u(t)\| \ << t^{\eps} \|u(0)\| \qquad t\to \infty
\end{equation}

which accounts for the logarithmic growth of \eqref{main}. Many papers obtained improved polynomial upper bounds for the growth rate, such as \cite{staffilani1997growth} \cite{colliander2001bilinear} \cite{zhong2008growth} \cite{catoire2008bounds} \cite{colliander2012remark} \cite{sohinger2011bounds} \cite{sohinger2011boundsb}

\quad

\myindent Secondly, our approach is close to the study of the linear case in \cite{faou2020weakly}. The question of growth of Sobolev norm for \textit{linear Schrödinger equation} perturbed by a smooth time-dependent potential has been widely studied, and generalized to the study of

\begin{equation}
    i\partial_t u = H u + P(t)u
\end{equation}

where $P(t)$ is a smooth time-dependent family of pseudodifferential operators of order $0$ and $H = - \Delta$ (torus) or $-\Delta + |x|^2$ (Euclidian space).

\myindent The seminal work of Bourgain \cite{bourgain1999growtha} establishes that when $P(t)$ is the multiplication by a \textit{quasiperiodic} potential on the torus a logarithmic growth of the Sobolev norms can be achieved, and it is optimal. Bourgain also studied the case of a random behaviour in time with certain smoothness conditions in \cite{bourgain1999growthb}. Regarding the growth rate, Bourgain proves that with a bounded smooth potential $V$ then the growth in any norm $H^s$ is bounded by $t^{\eps}$ for all $\eps > 0$ (with a constant that depends upon $s,V,\eps$) ; and that for a potential \textit{analytic} in time the bound can be refined to $(\log(t))^{\alpha}$. 

\myindent In the general case of $P(t)$ a family of pseudodifferential operators, Bambusi and Langella proved polynomial upper bounds on the growth rate in \cite{bambusi2022growth}. Assuming the potential to be \textit{periodic} in time, Maspero established in \cite{maspero2017time} abstract conditions under which there always exists unbounded solutions exhibiting polynomial growth. Loosening the time smoothness hypothesis, Erdogan, Killip and Schlag showed genericity of Sobolev norms growth when the potential is a stationary Markov process in \cite{erdougan2003energy}. See also \cite{eliasson2009reducibility}, \cite{delort2010growth},\cite{wang2008logarithmic}.

\quad

\myindent Finally, regarding potentials whose Sobolev norms decay with time, Raphael and Faou were able to exhibit logarithmic growth in the context of the linear Schrödinger equation on $\R^2$ with harmonic potential in \cite{faou2020weakly}. Their method, which we will adapt, relies on \textit{quasiconformal modulations} of so-called \textit{bubble} solutions of the unperturbed Harmonic Oscillator, followed by an approximation scheme. The growth mechanism is very similar to the example of \textit{Arnold diffusion} given in \cite{arnol2020instability} (see \cite{delshams2008geometric} for a review). Moreover, the author was also able to adapt their approximation scheme, used in the present paper, to exhibit logarithmic growth for the linear Schrödinger equation on the two dimensional torus with potential decaying to zero as $t\to \infty$ in \cite{chabert2023weakly}.

\subsection{\textit{Strategy of proof}}

\myindent We start by studying the equivalent of eigenfunctions for the nonlinear operator $Hu + u|u|^2$ which are so-called \textit{solitons}, defined as positive solutions to the time-independent elliptic equation

\begin{equation}\label{eqsolit}
    -\Delta Q + |x|^2 Q + Q|Q|^2 = \lambda Q
\end{equation}

for some $\lambda \in \R$. Their interest is that for each such soliton

\begin{equation}\label{oscers}
    u(t,x) := e^{-i\lambda t} Q(x)
\end{equation}

is a solution to the unperturbed equation \eqref{unperteq}. Section 3 is therefore devoted to establishing important results on solitons.

\myindent We will afterwards in section 4 \textit{modulate} the periodic solution \eqref{oscers}. Using the approach of \cite{faou2020weakly} for the linear case, the symmetries of equation \eqref{unperteq} yield the ansatz of solution, called a \textit{bubble}

\begin{equation}\label{ansatz1}
    u(t,x) = e^{-i\lambda s(t)} \frac{1}{L(t)} e^{-i\frac{b(t)|x|^2}{4 L(t)^2}} Q\left(\frac{x}{L(t)}\right)
\end{equation}

where $(L,b)$ are smooth real-valued functions and $s(t)$ is defined by $\frac{ds}{dt} = \frac{1}{L^2} > 0$ such that $s(t) \sim t$ when $t\to \infty$. There holds that $u$ is a solution to \eqref{unperteq} provided $(L,b)$ is a solution to a Hamiltonian dynamical system. This system is parameterized by its preserved \textit{action}

\begin{equation}
    E(L,b) := \frac{1}{L^2(t)}\left(\frac{b^2(t)}{4} + 1\right) + L^2(t)
\end{equation}

\myindent Now, the $H_x^1$ norm of the solution \eqref{ansatz1} is related to $E$ as we will prove

\begin{equation}
    \|u(t,\cdot)\|_{H_x^1}^2 \approx E(L,b) C(Q)
\end{equation}

where $C(Q)$ is a constant depending only of $Q$. Therefore, modulation enables us to build a solution $u$ with arbitrarily large $H_x^1$ norm from the soliton $Q$.

\myindent KAM theory allows for slightly perturbing the Hamiltonian dynamical system on $(L,b)$ to a quasi-Hamiltonian dynamical system where the energy $E(L,b)(t)$ grows logarithmically as $t\to \infty$. We will not detail the proof as we use directly a proposition from \cite{faou2020weakly}.

\myindent As we slightly perturb the Hamiltonian dynamical system, \eqref{ansatz1} doesn't solve \eqref{eqintro} exactly anymore. Thus, in section 5 and 6, we show that we may refine the ansatz to the exact solution

\begin{equation}
    u(t,x) = e^{-i\lambda s(t)} \frac{1}{L(t)} e^{-i\frac{b(t)|x|^2}{4 L(t)^2}} \left(Q\left(\frac{x}{L(t)}\right) + w\left(s(t), \frac{x}{L(t)} \right)\right)
\end{equation}

where $w(s,y)$ is a perturbation, that is its $H_x^1$ norm decays to zero as $s\to \infty$. Thus, the $H_x^1$ norm of $u$ is equivalent to $E(L,b)(t)^{\frac{1}{2}}$ which yields that it grows as $\left(\log(t)\right)^{\frac{1}{2}}$. The problem thus amounts to showing that there is a solution to the equation satisfied by $w$ which decays to zero at infinity. We will use a backward integration + Cauchy sequence scheme, that is we first build $w^M$ as solution such that $w^M(M) = 0$ at $s = M$. We will control uniformly in $M$ the decay of $w^M$ on $(s_0,M)$, and show that $w^M$ is a Cauchy sequence in an appropriate function space. Thus, we may extract its limit, which is a solution decaying to zero as $s \to \infty$. It is there that the need for a compensating potential $V$, which we may choose decaying to zero along with all its derivatives in the $H_x^k$ norms, will appear. Let us stress that the need for $V$ seems circumstantial to our particular proof, and thus one can hope to find a solution to the unperturbed equation \eqref{unperteq} exhibiting weak turbulence using similar techniques. 

\myindent Let us remark finally that adapting our proof should allow for logarithmic estimates on the growth of the $H_x^k$ norm of $u$ for all $k\geq 1$. However, this would greatly increase the technical complexity of the proof without any new original idea. Thus, we have limited ourselves to the $H_x^1$ norm as the growth mechanism is clearer.

\subsection{\textit{Acknowledgements}}

\myindent The author wishes to express her deepest thanks to Professor Pierre Germain and Professor Pierre Raphael, as the former helped greatly with formatting the present paper through proofreading and the latter inspired deeply the proof. The author is financially supported as a master's student by the Ecole Normale Supérieure (Paris, Ulm). The author reports there are no competing interests to declare.

\section{Preliminaries}
\myindent Throughout the rest of this paper, we will only consider radial functions, thus all the functions $u$ defined on $\R^2$ which will appear will be of the form $u(x) = f(|x|)$ for some $f$. For the purpose of simplicity we will not explicitly mention it when we define new function spaces. For example, we will denote by $L^2(\R^2;\C)$ the space of $\textit{radial}$ complex-valued functions over $\R^2$ with integrable squared absolute value.

\subsection{\textit{Notations}}
\myindent We denote by $L^2(\R^2 ; \R)$ the real Hilbert space of (radial) functions $\R^2 \to \R$ with integrable squared absolute value. We denote by $L^2(\R^2 ; \C)$ the space of (radial) functions $\R^2 \to \C$ such that $\int |f|^2 <  +\infty$. Unless otherwise stated, $L^2$ will designate $L^2(\R^2 ; \C)$. We introduce two scalar products over $L^2$ : first, the usual hermitian scalar product

\begin{equation}
    (f,g) := \int_{\R^2} f\overline{g}
\end{equation}

so $L^2$ may be viewed as a complex Hilbert space. However we can also see $f = f_1 + i f_2 \in L^2$ as a vector $\begin{pmatrix} f_1 \\ f_2 \end{pmatrix}$, its two components being in $L^2(\R^2 ; \R)$, and in that spirit we define

\begin{equation}
    \scal{f_1 + if_2}{g_1 + ig_2} := \int_{\R^2} f_1 g_1 + f_2 g_2
\end{equation}

so $L^2$ can be viewed as a real Hilbert space. As $\scal{\cdot}{\cdot}$ is just the real part of $(\cdot,\cdot)$, they both induce the same norm. For a complex-valued function $f$, we will always denote $f_1 := Re(f)$ and $f_2 := Im(f)$.

\subsection{\textit{Harmonic Oscillator and eigenfunctions}}

\myindent Following for example \cite{germain2016continuous} we know that $H = -\Delta + |x|^2$ is self-adjoint positive over $L^2$ for $(\cdot,\cdot)$, with compact resolvent, and we have explicitly

\begin{proposition}\label{diagH}
    There is a $L^2$ orthonormal basis of (radial) functions $(h_n)_{n\geq 0}$ such that 
    
    \begin{equation}
        Hh_n = 4n + 2
    \end{equation}
    
    \myindent Moreover, the $h_n$ are real, have $\mathcal{C}^{\infty}$ regularity, and they decay exponentially when $|x|\to \infty$. Finally, $h_0(x) = C \exp\left(-\frac{|x|^2}{2}\right)$ with $C$ a normalization constant
\end{proposition}

\myindent It is noteworthy that $(h_n)$ is an orthonormal basis of $L^2(\R^2;\R)$ as a real Hilbert space as well. Moreover, we see that $H$ has simple eigenvalues with fixed frequency gap, which ensures the following (see for example \cite[p.~123]{kato2013perturbation})

\begin{lemma} \label{spec}
    Let $E$ be one of the following function space : $L^2(\R^2;\R),L^2(\R^2;\C)$ (with hermitian scalar product). Let $A$ be a bounded self-adjoint operator over $E$ with operator norm $\|A\| < 2$. Then $H - 2 + A$ is a self-adjoint operator with compact resolvent, and there is an orthonormal basis of $E$ $\psi_0, \psi_1, ...$ and a sequence $\lambda_0 < \lambda_1 < ...$ such that 
    \begin{equation}
        (H-2+A)\psi_n = \lambda_n \psi_n
    \end{equation}
    
    \myindent Moreover there holds
    
    \begin{equation}
        \forall n\geq 0\qquad |\lambda_n - 4n| \leq \|A\|
    \end{equation}
    
\end{lemma}

\subsection{\textit{Function spaces}}

\myindent We remind that for $r\geq 0$ we set

\begin{equation}
    H_x^r := \{u\in L^2, H^{\frac{r}{2}}u\in L^2\}
\end{equation}

which is equipped with the norm 

\begin{equation}
    \|u\|_{H_x^r} := \|H^{\frac{r}{2}}u\|_{L^2}
\end{equation}

\myindent More precisely, if $u = \sum_{n\geq 0}\alpha_n h_n$ then there holds

\begin{equation}
    \|u\|_{H_x^r}^2 = \sum_{n\geq 0} (4n+2)^r |\alpha_n|^2
\end{equation}

\myindent Let us stress that these norms are non decreasing with respect to $r$. We also remind that the usual Sobolev spaces are defined for $r\in \R$ by 

\begin{equation}
    H^r := \{u\in L^2, (-\Delta + 1)^{\frac{r}{2}}u\in L^2\}
\end{equation}

and $H^r$ is equipped with the norm 

\begin{equation}
    \|u\|_{H^r} := \|(-\Delta + 1)^{\frac{r}{2}}u\|_{L^2}
\end{equation}

\myindent We will use basic properties of the $H_x^r$ norms whose proofs can be found in \cite[~pp. 6-7]{faou2020weakly}. First, in order to fix ideas, one has the norm equivalence 

\begin{equation} \label{equivalence}
    \exists c_r, C_r > 0 \qquad c_r\|f\|_{H_x^r} \leq \|f\|_{H^r} + \|<x>^r f\|_{L^2} \leq C_r\|f\|_{H_x^r}
\end{equation}

where $<x> := \sqrt{1 + |x|^2}$. Moreover, we cite three lemmas proved in \cite[~pp. 6-7]{faou2020weakly}

\begin{lemma}[$H_x^r$ is an algebra if $r> 1$]
    Let $r> 1$. Then there exists a constant $C_r > 0$ such that 
    
    \begin{equation}\label{algebra}
        \forall f,g \in H_x^r \qquad \|fg\|_{H_x^r} \leq C_r \|f\|_{H_x^r}\|g\|_{H_x^r}
    \end{equation}
    
\end{lemma}

\begin{lemma}
    For all $r \geq 0$ there exists a $C_r > 0$ such that for all $f \in H_x^{r+1}$ there holds
    
    \begin{equation}\label{xcarre}
        \||x|^2 f\|_{H_x^r} \leq C_r\|f\|_{H_x^{r+1}}
    \end{equation}
    
\end{lemma}

\begin{lemma}[Commutator formula]
    For all $r \geq 0$ there exists a $C_r > 0$ such that for all $f \in H_x^r$ 
    \begin{equation}\label{commut}
        \|[H^{\frac{r}{2}},|x|^2] f\|_{L^2} \leq C_r\|f\|_{H_x^r}
    \end{equation}
    
\end{lemma}

where we set $[A,B] := AB - BA$ the commutator of $A$ et $B$.

\myindent We finally prove the following lemma

\begin{lemma}[norm estimate]\label{contnorm2}
    Given $\eps > 0$ and $r > s$ there exists a $C_{\eps,r,s} > 0$ such that 
    
    \begin{equation}\label{estim1}
        \forall u \in H_x^r \qquad \|u\|_{H_x^s}^2 \leq \eps \|u\|_{H_x^r}^2 + C_{\eps,r,s} \|u\|_{L^2}^2
    \end{equation}
    
\end{lemma}

\begin{proof}
    Write $u = \sum_{n\geq 0} \alpha_n h_n$ so that $\|u\|_{H_x^t}^2 = \sum_{n\geq 0} (4n+2)^t |\alpha_n|^2$ for any $t \geq 0$. Find $N \geq 0$ so that whenever $n\geq N$ there holds $(4n+2)^s \leq \eps (4n+2)^r$. Now
    \begin{equation}
        \begin{split}
        \|u\|_{H_x^s}^2 &= \sum_{n\geq N} (4n+2)^s |\alpha_n|^2 + \sum_{n < N} (4n+2)^s |\alpha_n|^2 \\
        &\leq \eps \|u\|_{H_x^r}^2 + (4N+2)^s \|u\|_{L^2}^2
        \end{split}
    \end{equation}
\end{proof}

\subsection{\textit{Control of the $L^{\infty}$ norm}}
\myindent We remind the following result (where $H^r$ is the usual Sobolev space)

\begin{lemma}\label{Linfty}
    For all $r > 1$ the space $H^r$ is continuously embedded into $L^{\infty}(\R^2)$, that is there exists a $C_r > 0$ such that
    
    \begin{equation}
        \forall u\in H^r \qquad \|u\|_{L^{\infty}}\leq C_r\|u\|_{H^r}
    \end{equation}
    
\end{lemma}

\begin{proof}
    We remind the formula
    
    \begin{equation}
        \|u\|_{H^r}^2 = \int_{\R^2} |\hat{u}(\xi)|^2 <\xi>^{2r} d\xi
    \end{equation}
    
    where $\hat{u}$ is the Fourier transform of $u$. Now, we see that $\|u\|_{L^{\infty}} \leq C\|\hat{u}\|_{L^1}$ for a universal constant $C$, and we thus can write
    
    \begin{equation}
        \begin{split}
        \|u\|_{L^{\infty}} &\leq C\int_{\R^2} |\hat{u}(\xi)|d\xi\\
        &\leq C\left(\int_{\R^2} |\hat{u}(\xi)|^2 <\xi>^{2r} d\xi\right)^{\frac{1}{2}}\left(\int_{\R^2} \frac{d\xi}{<\xi>^{2r}} \right)^{\frac{1}{2}}\\
        &\leq C_r \|u\|_{H^r} 
        \end{split}
    \end{equation}
    
\end{proof}

\section{Solitons}

\myindent As stated in the introduction, the proof of theorem ~\ref{main} relies heavily on the study of \textit{solitons}, which are positive smooth solutions of the elliptic equation

\begin{equation}
    H Q + Q|Q|^2 = \lambda Q
\end{equation}

for some $\lambda \in \R$. In this section, we will prove important results on solitons.

\subsection{\textit{Existence, Uniqueness}}
\myindent This subsection is devoted to the demonstration of
\begin{proposition}\label{exunic}
    For all $\lambda > 2$ there exists a unique nonnegative and nontrivial solution $Q_{\lambda}$ to \eqref{eqsolit} in $H_x^1$. Moreover, $Q \in H_x^k$ for all $k\geq 1$. Finally, there holds $Q > 0$ on $\R^2$.
\end{proposition}

\myindent This proposition ensures in particular that $Q$ has $\mathcal{C}^{\infty}$ regularity over $\R^2$ (\cite{brezis2011functional}), and $Q(x)$ and all its derivatives decay to zero as $|x|\to \infty$.

\myindent Let us start the proof with a functional analysis lemma 

\begin{lemma}\label{injcomp}
    The function space $H_x^1$ is compactly embedded into $L^p(\R^2)$ for all $2 \leq p < + \infty$
\end{lemma}

\begin{proof}
    Following for example \cite[~p.281]{brezis2011functional}, we see that $H_x^1$ is continuously embedded into $L^p$ thanks to the Sobolev inequalities (indeed, the $H_x^1$ norm in finer than the $H^1$ Sobolev norm), and the Rellich-Kondrakov theorem \cite[~P.285]{brezis2011functional} yields compact embedding into $L^2_{loc}$. We remind that
    
    \begin{equation}
        \|u\|_{H_x^1}^2 = \|\nabla u\|_{L^2}^2 + \|xu\|_{L^2}^2
    \end{equation}
    
    \myindent Assume $u_n \rightharpoonup 0$ in $H_x^1$ (weak convergence). Then we have that $(u_n)$ is bounded in $H_x^1$. Find $M > 0$ such that $\forall n, \|u_n\|_{H_x^1}\leq M$. Choose $\delta > 0$. Then for a given $R > 0$ there holds 
    
    \begin{equation}
    \int_{|x| \geq R} |u_n|^2 \leq R^{-2}\|xu_n\|_{L^2}^2 \leq M^2 R^{-2} \leq \delta
    \end{equation}
    
    should we choose $R$ large enough. Now,  as $H_x^1$ is compactly embedded into $L^2(|x|\leq R)$, we know that 
    
    \begin{equation}
        \int_{|x|\leq R}|u_n|^2 \to 0
    \end{equation} 
    
    \myindent Thus, $\limsup \|u_n\|_{L^2}^2 \leq \delta$ for arbitrary $\delta > 0$ and therefore $u_n \to 0$ in $L^2$. We may conclude that $H_x^1$ is compactly embedded into $L^2$. 
    
    \myindent Now, given $2 \leq p < + \infty$, find $p < q < \infty$ and write $p = 2t + q(1-t)$ for some $t \in (0,1)$. Then we have thanks to Hölder's inequality 
    
    \begin{equation}
        \|u_n\|_{L^p} \leq \|u_n\|_{L^2}^t\|u_n\|_{L^q}^{1-t} \to 0
    \end{equation}
    
    \myindent Indeed, $(u_n)$ is bounded in $L^q$ as it is bounded in $H^1_x$. This yields the compactness of the embedding into $L^p$
\end{proof}

\myindent Let us now state a technical lemma

\begin{lemma}\label{contnorm1}
    Given $p > 2$ and $\eps > 0$ there exists a $C_{\eps,p} > 0$ such that for all $u\in H_x^1$ the holds 
    
    \begin{equation}
        \|u\|_{L^2}^2 \leq C_{\eps,p}\|u\|_{L^p}^2 + \eps \|u\|_{H_x^1}^2
    \end{equation}
    
\end{lemma}

\begin{proof}
    Write first
    
    \begin{equation}
        \int_{|x|\geq R} |u|^2 \leq R^{-2}\|u\|_{H_x^1}^2 \leq \eps \|u\|_{H_x^1}^2
    \end{equation}
    
    if $R > 0$ is big enough. Now, thanks to Hölder inequality 
    
    \begin{equation}
        \int_{|x|\leq R} |u|^2 \leq \left|\{|x|\leq R\}\right|^{1 - \frac{2}{p}} \|u\|_{L^p}^2 = C_{\eps}\|u\|_{L^p}^2
    \end{equation}
    
\end{proof}

\myindent Now, we are able to begin the proof of proposition \eqref{exunic}. We introduce the following functional over $H_x^1$ : for $u\in H_x^1$, set

\begin{equation}
    J(u) := \frac{1}{2}\|u\|_{H_x^1}^2 - \frac{\lambda}{2}\|u\|_{L^2}^2 + \frac{1}{4}\|u\|_{L^4}^4 = \int \frac{1}{2}\left(|\nabla u|^2 + |xu|^2\right) - \frac{\lambda}{2} |u|^2 + \frac{1}{4}|u|^4
\end{equation}

which is differentiable over $H_x^1$, its differential being the continuous linear form over $H_x^1$

\begin{equation}
    J'(u) = -\Delta u + |x|^2 u - \lambda u + u|u|^2
\end{equation}

where we can define, given $u\in H_x^1$, its distributional Laplacian by duality

\begin{equation}
    \forall g \in H_x^1, \qquad \scal{-\Delta u}{g}_{(H_x^1)' \times H_x^1} := (\nabla u,\nabla g)
\end{equation}

\myindent Our aim is to prove that $J$ realizes its minimum over $H_x^1$. Indeed if $J$ realizes its minimum at $u_0 \in H_x^1$, writing $J'(u_0) = 0$ yields that $u_0$ solves \eqref{eqsolit} in $(H_x^1)'$. 

\myindent Find a minimizing sequence $(u_n)\subset H_x^1$ for $J$. There holds

\begin{equation}
    J(u_n) \to \inf_{u\in H_x^1} J(u) \in [-\infty,+\infty)
\end{equation}

\myindent We see by choosing $\eps = \frac{1}{2\lambda}$ in lemma \eqref{contnorm1} that there exists a $C_{\eps} > 0$ such that 

\begin{equation}
    J(u)\geq \frac{1}{4}\|u\|_{H_x^1}^2 - \frac{\lambda C_{\eps}}{2}\|u\|_{L^4}^2 + \frac{1}{4}\|u\|_{L^4}^4 \geq \frac{1}{4}\|u\|_{H_x^1}^2 - C'
\end{equation}

where $C'$ is a constant. Indeed, the last inequality comes from the polynom $\frac{1}{4}x^4 - \frac{\lambda C_{\eps}}{2}x^2$ being bounded from below over $\R$.

\myindent Therefore, as $J(u_n)$ is bounded from above, $(u_n)$ is bounded in $H_x^1$. We can assume that $u_n$ converges weakly to a $Q \in  H_x^1$. Lemma \eqref{injcomp} ensures that $u_n \to Q$ strongly both in $L^2$ and in $L^4$. Moreover, 

\begin{equation}
    \|Q\|_{H_x^1} \leq \liminf \|u_n\|_{H_x^1}
\end{equation}

\myindent We therefore infer that 

\begin{equation}
    J(Q) \leq \liminf J(u_n) = \inf_{u\in H_x^1} J(u)
\end{equation}

which ensures $J$ realizes its minimum at $Q$. 

\myindent Now, if $u\in H_x^1$ then $|u| \in H_x^1$ and $J(|u|) \leq J(u)$. Indeed, this follows instantly from

\begin{lemma}[Convexity of the Dirichlet functional]\label{dirich}
    Let $u \in H^1$ (usual Sobolev space). Then $|u|\in H^1$ and there holds
    
    \begin{equation}
        \int_{\R^2} |\nabla |u| |^2 \leq \int_{\R^2} |\nabla u|^2
    \end{equation}
    
\end{lemma}

\begin{proof}
    By approximation, we can assume that $u$ is regular. 
    
    \myindent Decompose $u$ in real and imaginary parts $u = f + ig$. Then a.e.
    
    \begin{equation}
        \nabla|u| = \nabla \sqrt{f^2 + g^2} = \frac{f\nabla f + g\nabla g}{\sqrt{f^2 + g^2}}
    \end{equation}
    \myindent Hence
    
    \begin{equation}
        \begin{split}
        \int |\nabla |u||^2 &= \int \frac{|f\nabla f + g\nabla g|^2}{f^2 + g^2} \\
        &= \int \frac{1}{f^2 + g^2} \left[f^2 |\nabla f|^2 + g^2 |\nabla g|^2 + 2 fg \nabla f \cdot \nabla g\right]\\
        &= \int |\nabla f|^2 + \int |\nabla g|^2 - \int \frac{|g\nabla f - f \nabla g|^2}{f^2 + g^2}
        \end{split}
    \end{equation}
    
    which yields the result.
\end{proof}

\myindent This proves that $|Q|\in H_x^1$, and moreover that $J(|Q|) \leq J(Q)$, and the equality holds by the minimization property. Therefore we can assume $Q$ to be real and nonnegative. Moreover, we can prove that $Q$ is non zero by showing that $J(Q) < 0$. Because $Q$ is a minimizer of $J$ it suffices to show that $J$ isn't nonnegative. Try 

\begin{equation}
    h_0(x) := \exp(-\frac{1}{2}|x|^2)
\end{equation}

 which satisfies $Hh_0 = 2h_0$. There holds

\begin{equation}
    J(th_0) = \frac{1}{2}t^2 (2-\lambda)\|h_0\|_{L^2}^2 + \frac{1}{4} t^4 \|h_0\|_{L^4}^4 
\end{equation}

\myindent Should we choose $t > 0$ small enough, the right-hand side is negative. This shows that $\inf_{H_x^1} J < 0$ which in turn ensures that $J(Q) < 0$. Let us stress here that we use the hypothesis $\lambda > 2$. Indeed, one can easily prove that \eqref{eqsolit} doesn't have a nontrivial solution in $H_x^1$ when $\lambda \leq 2$.

\myindent We then write that $J'(Q) = 0$ in $(H_x^1)'$ (because of minimization), which reads

\begin{equation}
    -\Delta Q + |x|^2 Q + Q^3 - \lambda Q = 0 \qquad \text{in} \ (H_x^1)'
\end{equation}

\myindent Now, this ensures that $H Q \in L^2$, which means $Q \in H_x^2$. Lemma \eqref{algebra} yields that $H_x^k$ is an algebra for $k\geq 2$ so $H Q$ is also in $H_x^2$, meaning $Q\in H_x^4$. We can iterate and see that 

\begin{equation}
    Q \in \bigcap_{k\geq 1} H_x^k
\end{equation}

\myindent Thus, Sobolev injections yield that $Q$ has $\mathcal{C}^{\infty}$ regularity over $\R^2$ and that $Q(x)$ along with all its derivative decay to zero as $|x|\to \infty$. Thus we have proved the existence part of \eqref{exunic}.

\quad

\myindent We now prove that $Q > 0$ on $\R^2$. Indeed, recall that $Q$ is radial so we may set $Q(x) = f(|x|)$ for some smooth $f$ on $[0,+\infty)$. $f$ is a nonnegative solution to a second order ordinary differential equation on $(0,+\infty)$. Assume there exists $r_0 > 0$ such that $f(r_0) = 0$. There holds $f'(r_0) = 0$ (as $f \geq 0$), thus Cauchy theorem yields that $f = 0$ on a neighborhood of $r_0$. Therefore, the set $\{f = 0\}$ is both open and closed in $(0,+\infty)$ thus it is either empty, in which case $Q = 0$ which is a contradiction, either it equals $(0,+\infty)$. It remains to be proven that $Q(0) \neq 0$ for which we cannot apply directly Cauchy theorem as we loose local Lipschitzness at $r = 0$. However, assume $Q(0) = 0$. Then, in a neighborhood $U$ of $0$, there holds $\lambda - |x|^2 - Q^2 > 0$ and thus

\begin{equation}
    \Delta (-Q) = (\lambda - |x|^2 - Q^2) Q \geq 0
\end{equation}

\myindent The maximum principle ensures that $-Q$ cannot have a strict local maximum at $0$ which is a contradiction. Therefore we have proven that $Q > 0$ on $\R^2$.

\quad

\myindent In order to conclude the proof of \eqref{exunic} we need to show the uniqueness of $Q$. The proof relies on the following lemma

\begin{lemma}\label{eqsolv}
    Given $f \in H_x^1$ nonnegative, there exists a nonnegative solution $u \in H_x^1$ to the equation
    
    \begin{equation}
        Hu + u|u|^2 = f
    \end{equation}
\end{lemma}

\begin{proof}
    Given $u \in H_x^1$, set
    
    \begin{equation}
        J_f(u) := \frac{1}{2}\int |\nabla u|^2 + |\cdot|^2 |u|^2  + \frac{1}{4}\int |u|^4 - \int fu = \frac{1}{2}\|u\|_x^2 + \frac{1}{4}\|u\|_4^4 - \scal{f}{u}
    \end{equation}
    
    \myindent We see that $J_f$ is differentiable over $H_x^1$, with differential 
    \begin{equation}
        J_f'(u) = Hu + u|u|^3 - f \in (H_x^1)'
    \end{equation}
    
    \myindent Moreover, $J_f$ is bounded from below as there holds
    
    \begin{equation}
        J_f(u) \geq \frac{1}{2}\|u\|_{H_x^1}^2 + \frac{1}{4}\|u\|_{L^4}^2 - \|f\|_{L^{\frac{4}{3}}}\|u\|_{L^4} \geq \inf_{x\in \R} \left(\frac{1}{4}x^2 - \|f\|_{L^{\frac{4}{3}}} x\right) > -\infty
    \end{equation}
    
    \myindent Find $(u_n)\subset H_x^1$ a minimizing sequence for $J_f$. $(u_n)$ is bounded in $H_x^1$ thus, without loss of generality, converges weakly towards $u \in H_x^1$. Lemma \eqref{injcomp} ensures that $u_n \to u$ in $L^2$ and in $L^4$. We get using $\|u\|_{H_x^1}\leq \liminf \|u_n\|_{H_x^1}$ that
    
    \begin{equation}
        J_f(u) \leq \liminf J_f(u_n) = \inf_{H_x^1} J_f
    \end{equation}
    
    which proves that $u$ minimizes $J_f$ over $H_x^1$. Using lemma \eqref{dirich} and $f\geq0$ there holds $J_f(|u|) \leq J_f(u)$ so we can always assume that $u \geq 0$. Write finally $J_f'(u) = 0$
    \end{proof}

\myindent Suppose now that there are two nontrivial solutions $u_1\geq 0$ and $u_2\geq 0$ to equation \eqref{eqsolit}. Set $\overline{u} = u_1 + u_2$. We then have

\begin{equation}
    H\bar{u} -\lambda \bar{u} + \bar{u}^3 \geq H\bar{u} -\lambda \bar{u} + u_1^3 + u_2^3 = 0
\end{equation}

\myindent We set $w_0 = \overline{u}$ then recursively find, using \eqref{eqsolv}, $w_{n+1}\in H_x^1$ a nonnegative solution to

\begin{equation}
    Hw_{n+1} + w_{n+1}^3 = \lambda w_n
\end{equation}

\myindent We show recursively that for all $n$ there holds (in the distributional sense therefore almost everywhere)

\begin{equation}
    u_i \leq w_{n+1} \leq w_n \leq \overline{u}\qquad i = 1,2
\end{equation}

\myindent Indeed, we can write

\begin{equation}
    Hw_1 + w_1^3 = \lambda w_0 \leq Hw_0 + w_0^3
\end{equation}

\myindent Test this inequality against the nonnegative $H_x^1$ function $(w_1 - w_0)^+$ (where $x^+ := \max(0,x)$)

\begin{equation}
    \int |\nabla(w_1 - w_0)^+|^2 + |\cdot|^2 |(w_1 - w_0)^+|^2 + |w_1^3 - w_0^3|(w_1 - w_0)^+ \leq 0
\end{equation}

\myindent Necessarily there holds $(w_1 - w_0)^+ = 0$ i.e. $w_1 \leq w_0$. We can also write 

\begin{equation}
    Hw_1 + w_1^3 \geq \lambda u_i = Hu_i + u_i^3
\end{equation}

\myindent Testing this inequality against the nonnegative $H_x^1$ function $(u_i - w_1)^+$ yields the other inequality. We can iterate this reasoning because there holds 

\begin{equation}
    Hw_1 + w_1^3 -\lambda w_1 = \lambda(w_0 - w_1) \geq 0
\end{equation}

\myindent Let a subsequence of $(w_n)$ weakly converge to a $w \in H_x^1$ : there holds $Hw - \lambda w + w^3 = 0$ in $(H_x^1)'$, and $u_i\leq w$ for $i = 1,2$. Therefore,

\begin{equation}
    \begin{split}
    \scal{Hu_i}{w} -\lambda \scal{u_i}{w} + \int u_i^3 w &= 0\\
    \scal{Hw}{u_i} -\lambda \scal{w}{u_i} + \int w^3 u_i &= 0
    \end{split}
\end{equation}

\myindent Subtracting yields 

\begin{equation}
    \int u_i w (w^2 - u_i^2) = 0
\end{equation}

the integrand being nonnegative. Moreover, as we have proven that there holds $u_i > 0$, we can conclude that $u_i = w$ for $i = 1,2$ and thus that $u_1 = u_2$.
 
\subsection{\textit{Differentiability in $\lambda$}}

\myindent Let us study the regularity of the function 

\begin{equation}
    \lambda \in (2, +\infty) \to Q_{\lambda} \in H_x^1
\end{equation}

where $Q_{\lambda}$ is the unique positive solution to $HQ_{\lambda} + Q_{\lambda}^3 = \lambda Q_{\lambda}$ given by \eqref{exunic}.

\myindent We will prove the following result

\begin{proposition}\label{deriv}
    The function $\lambda \to Q_{\lambda}$ is differentiable as a function from $(2, +\infty)$ to $H_x^1$
\end{proposition}

\myindent It is noteworthy that the proof can easily be iterated, showing the $\mathcal{C}^{\infty}$ regularity of this function. However, we will not need this result.

\myindent Let us begin with

\begin{lemma}\label{normincr}
    The function $\lambda \to \|Q_{\lambda}\|_{L^2}$ is increasing over $(2,+\infty)$
\end{lemma}

\begin{proof}
    Given its construction, $Q_{\lambda}$ is the unique positive minimizer of 
    
    \begin{equation}
        J(\lambda) : u \to \|u\|_{H_x^1}^2 + \frac{2}{4} \|u\|_{L^4}^4 - \lambda \|u\|_{L^2}^2
    \end{equation}
    
    \myindent Therefore we can write, given $\mu < \lambda$ 
    
    \begin{equation}
    \begin{split}
    \|Q_{\mu}\|_{H_x^1}^2 + \frac{2}{4}\|Q_{\mu}\|_{L^4}^4 - \mu \|Q_{\mu}\|_{L^2}^2 &< \|Q_{\lambda}\|_{H_x^1}^2 + \frac{2}{4}\|Q_{\lambda}\|_{L^4}^4 - \mu\|Q_{\lambda}\|_{L^2}^2\\
    &=  \|Q_{\lambda}\|_{H_x^1}^2 + \frac{2}{4}\|Q_{\lambda}\|_{L^4}^4 - \lambda\|Q_{\lambda}\|_{L^2}^2 + (\lambda - \mu)\|Q_{\lambda}\|_{L^2}^2 \\
    &< \|Q_{\mu}\|_{H_x^1}^2 + \frac{2}{4}\|Q_{\mu}\|_{L^4}^4- \lambda \|Q_{\mu}\|_{L^2}^2 + (\lambda - \mu)\|Q_{\lambda}\|_{L^2}^2
    \end{split}
\end{equation}

which reads 

\begin{equation}
    \forall \mu < \lambda\qquad (\lambda - \mu)(\|Q_{\lambda}\|_{L^2}^2 - \|Q_{\mu}\|_{L^2}^2) > 0
\end{equation}
\end{proof}

\myindent Thus, as $\|Q_{\lambda}\|_{H_x^1}^2 + \|Q_{\lambda}\|_{L^4}^4 = \lambda \|Q_{\lambda}\|_{L^2}^2$, we see that $\lambda \to Q_{\lambda}$ is locally bounded in $H_x^1$ norm.

\quad

\myindent Now, we prove that $\lambda \to Q_{\lambda}$ is continuous. Given $\lambda > 2$, find $\lambda_n \to \lambda$. We know that $(Q_{\lambda_n})$ is a bounded sequence in $H_x^1$, therefore it is weakly compact. If $f \in H_x^1$ is a weak limit point, there holds in $(H_x^1)'$ 

\begin{equation}
    -\Delta f + |x|^2 f + f^3 = \lambda f
\end{equation}

and $f$ is nonnegative. The uniqueness proved above ensures that $f$ is either equal to $Q_{\lambda}$ or $0$. However the latter cannot happen because lemma \eqref{injcomp}  yields that $\|f\|_{L^2}^2 = \lim \|Q_{\lambda_n}\|_{L^2}^2 \geq \|Q_{\lambda - \eps}\|_{L^2}^2 > 0$ (we use lemma \eqref{normincr}). This ensures that the only limit point of $Q_{\lambda_n}$ is $Q_{\lambda}$, thus 

\begin{equation}
    Q_{\lambda_n} \rightharpoonup Q_{\lambda} \qquad \text{weakly in}\ H_x^1
\end{equation}

\myindent Thanks to lemma \eqref{injcomp}, there holds moreover

\begin{equation}
    Q_{\lambda_n}^3 - \lambda_n Q_{\lambda_n} \to Q_{\lambda}^3 - \lambda Q_{\lambda}\qquad \text{strongly in} \ L^2
\end{equation}

\myindent However we can write 

\begin{equation}
    Q_{\lambda_n} = H^{-1}(Q_{\lambda_n}^3 -  \lambda_nQ_{\lambda_n})
\end{equation}

and using the continuity of $H^{-1}$ from $L^2$ into $H_x^1$ there holds

\begin{equation}
    Q_{\lambda_n} \to Q_{\lambda}\qquad \text{strongly in} \ H_x^1
\end{equation}

thus concluding that $\lambda \to Q_{\lambda}$ is continuous from $(2,+\infty)$ to $H_x^1$

\quad

\myindent Let us now establish that $\lambda \to Q_{\lambda}$ is differentiable.

\myindent We will use the following lemma

\begin{lemma}\label{uniqueness}
    Suppose that $u\in H_x^1$ is a solution to 
    
    \begin{equation}
        -\Delta u + |x|^2 u + 3Q_{\lambda}^2 u = \lambda u \qquad \text{in}\ (H_x^1)'
    \end{equation}
    
    \myindent Then necessarily $u = 0$
\end{lemma}

\myindent It is rather interesting that with the notations we will define below, for $\lambda = 2 + \eps$ with $\eps > 0$ small enough, this lemma will be a straightforward consequence of the inversibility of the operator $H_+ := H + 3Q_{\lambda}^2 - \lambda$, which we will prove using spectral theory. The proof here is of a very different nature and holds when $\lambda$ isn't near $2$.

\begin{proof}
    Take $u\in H_x^1$ a solution. As the equation is the same on the real and imaginary parts of $u$ we can assume $u$ is real-valued. Then, $u_+ := \max(0,u) \in H_x^1$ (indeed it is enough to know that $|u|\in H_x^1$, see \eqref{dirich}). Moreover, as distributions
    
    \begin{equation}\label{kato}
        \Delta(u_+) \geq 1_{u > 0} \Delta u
    \end{equation}
    
    where $1_{u> 0} = 0$ if $u \leq 0$ and $1$ ortherwise. Let us postpone the proof of this inequality. There holds
    
    \begin{equation}
    \begin{split}
    -\Delta\left(u_+\right) &\leq 1_{u > 0} \left(-\Delta u\right) \\
    &= 1_{u > 0}\left(-|\cdot|^2 u-3Q_{\lambda}^2 u+ \lambda u\right) \\
    &= -|\cdot|^2 u_+ -3Q_{\lambda}^2 u_+ + \lambda u_+
    \end{split}
\end{equation}

\myindent We can apply this distributional inequality to the positive function $Q_{\lambda}$ :

\begin{equation}
    \begin{split}
    \lambda\scal{u_+}{Q_{\lambda}} - 3 \scal{Q_{\lambda}^2 u_+}{Q_{\lambda}} &\geq \scal{H(u_+)}{Q_{\lambda}} \\
    &= \scal{u_+}{HQ_{\lambda}}\\
    &= \lambda \scal{u_+}{Q_{\lambda}} - \scal{u_+}{Q_{\lambda}^3}
    \end{split}
\end{equation}

which yields

\begin{equation}
    2\int_{\R^2} u_+ Q_{\lambda}^3 \leq 0
\end{equation}

\myindent As $Q_{\lambda}$ is positive and $u_+$ nonnegative, this ensures $u_+ = 0$ i.e.

\begin{equation}
    u \leq 0
\end{equation}

\myindent We then write 

\begin{equation}
    \begin{split}
    \lambda\scal{u}{Q_{\lambda}} - 3 \scal{Q_{\lambda}^2 u}{Q_{\lambda}} &= \scal{H(u)}{Q_{\lambda}} \\
    &= \scal{u}{HQ_{\lambda}}\\
    &= \lambda \scal{u}{Q_{\lambda}} - \scal{u}{Q_{\lambda}^3}
    \end{split}
\end{equation}

which reads

\begin{equation}
    2\int_{\R^2} uQ_{\lambda}^3 = 0
\end{equation}

\myindent Thus we finally have $u = 0$
\end{proof}

\myindent We now prove inequality \eqref{kato}. Following Kato's inequality \cite{kato1972schrodinger} we know that 

\begin{equation}
    \Delta |u| \geq sgn(u)\Delta(u)
\end{equation}

\myindent Where $sgn(u)$ is the sign of $u$, i.e. $-1$ when $u< 0$, $0$ when $u = 0$, and $1$ when $u > 0$. We can write $u_+ = \frac{1}{2}\left(u + |u|\right)$ and conclude using $\frac{1}{2}(1 + sgn(u)) = 1_{u > 0}$.

\myindent We are now in position to prove that $\lambda \to Q_{\lambda}$ is differentiable. Given $\lambda > 2$, let us consider for $h$ small (for example $h \in ]-\eps, \eps[$ with $\lambda - \eps > 2$)

\begin{equation}
    \tau(h) := \frac{Q_{\lambda + h} - Q_{\lambda}}{h}
\end{equation}

which solves

\begin{equation}
    -\Delta \tau(h) + |x|^2 \tau(h) + (Q_{\lambda + h}^2 + Q_{\lambda + h}Q_{\lambda} + Q_{\lambda}^2) \tau(h) = \lambda \tau(h) + Q_{\lambda + h}
\end{equation}

\myindent Suppose first that $\tau(h)$ is not bounded in $H_x^1$ when $h \to 0$. Then one can find $h_k\to 0$ such that

\begin{equation}
    \|\tau(h_k)\|_{H_x^1}\to \infty
\end{equation}

\myindent There holds

\begin{equation}
    \begin{split}
    \|\tau(h_k)\|_{H_x^1}^2 &\leq  \|\tau(h_k)\|_{H_x^1}^2 + \int_{\R^2} (Q_{\lambda + h_k}^2 + Q_{\lambda + h_k} Q_{\lambda} + Q_{\lambda}^2)\tau(h_k)^2 \\
    &= \lambda \|\tau(h_k)\|_{L^2}^2 + \int_{\R^2} Q_{\lambda + h_k} \tau(h_k)\\
    &\leq \|\tau(h_k)\|_{L^2}(\lambda\|\tau(h_k)\|_{L^2} + C)
    \end{split}
\end{equation}

where $C = \sup_k \|Q_{\lambda + h_k}\|_{L^2}$. This ensures that 

\begin{equation}
\|\tau(h_k)\|_{L^2} \to \infty
\end{equation}

and thus there exists a $C>0$ such that 

\begin{equation}
    \|\tau(h_k)\|_{H_x^1} \leq C \|\tau(h)\|_{L^2}
\end{equation}

\myindent Now, this ensures that $\frac{\tau(h_k)}{\|\tau(h_k)\|_{L^2}}$ is bounded in $H_x^1$ when $k\to \infty$, therefore one can find a weak limit point $u \in H_x^1$. Using $\frac{Q_{\lambda + h_k}}{\|\tau(h_k)\|_{L^2}} \to 0$ in $H_x^1$ one finds

\begin{equation}
    -\Delta u + |x|^2 u + 3 Q_{\lambda}^2 u = \lambda u \qquad \text{in} \ (H_x^1)'
\end{equation}

\myindent Thus using lemma \eqref{uniqueness} we know that $u = 0$. However using lemma \eqref{injcomp} we should have $\|u\|_{L^2} = 1$ which is a contradiction. We therefore infer that $\tau(h)$ is bounded in $H_x^1$ as $h \to 0$, thus is weakly compact in $H_x^1$.

\myindent Now, suppose that $h_n \to 0$ is such that

\begin{equation}
    \tau(h_n) \rightharpoonup g \qquad \text{weakly in} \ H_x^1
\end{equation}

\myindent We have that $g$ is a distributional solution (thus a solution in $(H_x^1)'$) of 

\begin{equation}\label{eqderiv0}
    H g + 3 Q_{\lambda}^2 g = \lambda g + Q_{\lambda}
\end{equation}

\myindent Lemma \eqref{uniqueness} ensures that this equation has at most one solution in $H_x^1$, thus $\tau(h)$ is weakly compact and have a unique weak limit point in $H_x^1$ when $h \to 0$. This ensures that 

\begin{equation}
    \tau(h) \rightharpoonup g \qquad \text{weakly in} \ H_x^1 \ h\to 0
\end{equation}

where $g$ is the unique solution to \eqref{eqderiv0}. Using now the compactness lemma \eqref{injcomp}, the continuity of $H$ from $L^2$ to $H_x^1$, and the equation 

\begin{equation}
    \tau(h) = H^{-1}\left(-(Q_{\lambda + h}^2 + Q_{\lambda + h}Q_{\lambda} + Q_{\lambda}) \tau(h) + \lambda \tau(h) + Q_{\lambda + h}\right)
\end{equation}

\myindent We infer that the convergence $\tau(h) \to g$ happens strongly in $H_x^1$, thus concluding the proof of proposition \eqref{deriv}. We also have the explicit formula 

\begin{lemma}[characterization of $\partial_{\lambda}Q_{\lambda}$]\label{eqderiv}
    The derivative $\partial_{\lambda}Q_{\lambda}$ is the unique solution $u$ in $H_x^1$ to the equation 
    
    \begin{equation}
        (H + 3Q_{\lambda}^2 - \lambda) u = Q_{\lambda}
    \end{equation}
    
\end{lemma}

\subsection{\textit{Bifurcation branch near $\lambda = 2$}}

\myindent We give an equivalent in $H_x^1$ of the soliton $Q_{2+\eps}$ as $\eps \to 0$

\begin{proposition}[Linear approximation near $\lambda = 2$]\label{linapprox}
    Define $h_0(x) = C\exp\left(-\frac{|x|^2}{2}\right)$ with $C$ such that $\|h_0\|_{L^2} = 1$. Then as $\eps \to 0$ there holds in $H_x^1$
    
    \begin{equation}
        Q_{2 + \eps} = \eps^{\frac{1}{2}} \|h_0\|_{L^4}^{-2} h_0 + o(\eps^{\frac{1}{2}})
    \end{equation}
\end{proposition}

\begin{proof}
We first notice that 

\begin{equation}
    Q_{2 + \eps} \to 0 \qquad \text{in}\ H_x^1
\end{equation}

\myindent Indeed, we know that 

\begin{equation}
    \|Q_{2 + \eps}\|_{H_x^1}^2 \leq (2 + \eps) \|Q_{2 + \eps}\|_{L^2}^2
\end{equation}

and the right-hand side is bounded as $\eps \to 0$ thank to \eqref{normincr}. Given $u$ a weak limit point of $Q_{2 + \eps}$ in $H_x^1$ when $\eps \to 0$, we see that

\begin{equation}
    Hu + u^3 = 2 u
\end{equation}

\myindent This in turns implies that

\begin{equation}
    \scal{Hu}{u} - 2\|u\|_{L^2}^2 = - \int |u|^4
\end{equation}

\myindent As \eqref{diagH} yields $\scal{Hu}{u} \geq 2\|u\|_{L^2}^2$ this implies $u = 0$. Therefore $Q_{2 + \eps}$ converges weakly to $0$ in $H_x^1$ and using the same trick as before (compactness of injection and continuity of $H^{-1} : L^2 \to H_x^1$) this convergence happens in $H_x^1$ norm.

\myindent Now we study the convergence of $\frac{Q_{2 + \eps}}{\|Q_{2 + \eps}\|_{H_x^1}}$. There holds

\begin{equation}
    H\left(\frac{Q_{2 + \eps}}{\|Q_{2 + \eps}\|_{H_x^1}}\right)  + \frac{Q_{2 + \eps}^3}{\|Q_{2 + \eps}\|_{H_x^1}} = (2 + \eps) \frac{Q_{2 + \eps}}{\|Q_{2 + \eps}\|_{H_x^1}}
\end{equation}

\myindent Therefore if $u\in H_x^1$ is a weak limit point to $\frac{Q_{2 + \eps}}{\|Q_{2 + \eps}\|_{H_x^1}}$ when $\eps \to 0$ there holds

\begin{equation}
    H u \leq 2u
\end{equation}

which implies that $u$ is proportional to $h_0$. Now,

\begin{equation}
    \left(\frac{\|Q_{2 + \eps}\|_{H_x^1}}{\|Q_{2 + \eps}\|_{L^2}}\right)^2 + \frac{\|Q_{2 + \eps}\|_{L^4}^4}{\|Q_{2 + \eps}\|_{L^2}^2} = 2 + \eps
\end{equation}

\myindent Using moreover that 

\begin{equation}
    \frac{\|Q_{2 + \eps}\|_{L^4}^4}{\|Q_{2 + \eps}\|_{L^2}^2} \leq C \|Q_{2 + \eps}\|_{H_x^1}^2 \to 0
\end{equation}

we now find that 

\begin{equation}
    \frac{\|Q_{2 + \eps}\|_{H_x^1}}{\|Q_{2 + \eps}\|_{L^2}} \to \sqrt{2} \qquad \eps \to 0
\end{equation}

and the compact injection of $H_x^1$ into $L^2$ \eqref{injcomp} yields 

\begin{equation}
    \|u\|_{L^2} = \frac{1}{\sqrt{2}}
\end{equation}

\myindent Therefore, $\frac{Q_{2 + \eps}}{\|Q_{2 + \eps}\|_{H_x^1}}$ converges weakly to its unique weak limit point $\frac{1}{\sqrt{2}} h_0$ in $H_x^1$, and this convergence happens in $H_x^1$ norm with again the same trick.

\myindent We already know

\begin{equation}
    \frac{Q_{2 + \eps}}{\|Q_{2 + \eps}\|_{H_x^1}} \to \frac{1}{\sqrt{2}} h_0 \qquad \text{in} \ H_x^1
\end{equation}

\myindent Using that $HQ_{2+\eps} + Q_{2 + \eps}^3 = (2 + \eps)Q_{2 + \eps}$ and that $Q_{2+ \eps} = \|Q_{2 + \eps}\|_{L^2}h_0 + o(\|Q_{2+\eps}\|_{L^2})$ in $H_x^1$ we can write

\begin{equation}
    \begin{split}
    &\left\| \|Q_{2 + \eps}\|_{L^2} h_0 + o(\cdot)\right\|_{H_x^1}^2 +\left\| \|Q_{2 + \eps}\|_{L^2} h_0 + o(\cdot)\right\|_{L^4}^4 = (2+\eps)\left\|\|Q_{2 + \eps}\|_{L^2} h_0 + o(\cdot)\right\|_{L^2}^2\\
    &\text{thus}\qquad 2\|Q_{2 + \eps}\|_{L^2}^2 + \|Q_{2 + \eps}\|_{L^2}^4 \|h_0\|_{L^4}^4 = (2 + \eps)\|Q_{2 + \eps}\|_{L^2}^2 + o(\cdot)\\
    &\text{which yields} \qquad \|Q_{2 + \eps}\|_{L^2} = \eps^{\frac{1}{2}} \|h_0\|_{L^4}^{-2} + o(\cdot)
    \end{split}
\end{equation}

\end{proof}

\subsection{\textit{$L^{\infty}$ bound}}

\myindent We will need to use, at some point, that the multiplication by $Q_{2+\eps}^2$ is continuous on the function spaces $H_x^k$, and with operator norm arbitrarily small as $\eps \to 0$. To this effect, we prove the following statement

\begin{proposition}\label{infnorm}
    Let $k \geq 0$. There exists a constant $C_k > 0$ such that when $\eps \to 0$ there holds the bound 
    
    \begin{equation}
        \|D^k Q_{2+\eps}\|_{L^{\infty}} \leq C_k \eps^{\frac{1}{2}}
    \end{equation}
\end{proposition}

\myindent Let us stress that this bound will typically happen over an open interval $\eps \in (0, \alpha_k)$ where the upper bound $\alpha_k$ depends of the derivation order $k$ that we consider. However we will always consider only a finite number of derivatives together so we will always be able to have a uniform interval where the bound holds. That is why we will write "when $\eps \to 0$" without specifying the meaning of this assertion : it will mean "there exists a $\alpha > 0$ such that on $(0,\alpha)$ there holds".

\begin{proof}
    Using lemma \eqref{Linfty}, there holds 
    
    \begin{equation}
        \|D^k Q\|_{L^{\infty}} \leq C_k\|D^k Q\|_{H^2} \leq C_k\|Q\|_{H^{2 + k}}
    \end{equation}
    
    \myindent Thus, using lemma \eqref{equivalence}, we need only prove that $\|Q\|_{H_x^{k+2}} \leq C_k \eps^{\frac{1}{2}}$. As however the $H_x^r$ norms are increasing with respect to $r$, we need only prove that given an even integer $2m$ there holds 
    
    \begin{equation}
        \|Q\|_{H_x^{2m}}\leq C_m \eps^{\frac{1}{2}}
    \end{equation}
    
    \myindent Thanks to \eqref{linapprox}, we know the statement for $m = 0$. We also see that using \eqref{injcomp} and \eqref{linapprox}
    
    \begin{equation}
        \begin{split}
        \|Q^3\|_{L^2} &= \|Q\|_{L^6}^3\\
        &\leq C\|Q\|_{H_x^1}^3 \\
        &\leq C \eps^{\frac{3}{2}} \\
        &\leq C\eps^{\frac{1}{2}}
        \end{split}
    \end{equation}
    
    provided we suppose $\eps$ to be small enough.
    \myindent Thus, we can write
    
    \begin{equation}
        \begin{split}
        \|Q\|_{H_x^2} &= \|HQ\|_{L^2} \\
        &= \left\|\lambda Q - Q^3\right\|_{L^2}\\
        &\leq C_2 \eps^{\frac{1}{2}}
        \end{split}
    \end{equation}
    
    \myindent Now, assume $\|Q\|_{H_x^{2m}} \leq C_m \eps^{\frac{1}{2}}$ for a given $m \geq 1$. Then we have provided for example $1 > \eps$
    
    \begin{equation}
        \begin{split}
        \|Q\|_{H_x^{2m + 2}} &= \|HQ\|_{H_x^{2m}} \\
        &= \left\|\lambda Q - Q^3\right\|_{H_x^{2m}}\\
        &\leq 3C_m\eps^{\frac{1}{2}} + C\|Q\|_{H_x^{2m}}^3 \\
        &\leq C_{m + 1} \eps^{\frac{1}{2}}
        \end{split}
    \end{equation}
    
    where we use lemma \eqref{algebra} to say that $H_x^{2m}$ is an algebra.
\end{proof}

\section{Modulation}

\myindent In this section we introduce the \textit{modulation operators} that we will use in order to transform equation \eqref{eqintro} into a perturbation of equation \eqref{eqsolit}. Moreover, we detail how those modulation act upon the $H_x^1$ norm of solutions and how we obtain norm growth using the theory of quasi-Hamiltonian dynamical systems.

\subsection{\textit{Modulation operators and energy estimates}}

\myindent Given $u \in H_x^1$ and $N > 0$ we set 

\begin{equation}
    (S_N u)(x) := \frac{1}{N} u\left(\frac{x}{N}\right)
\end{equation}

\myindent As we are in dimension 2, 

\begin{equation}
    \|S_N u\|_{L^2} = \|u\|_{L^2}
\end{equation}

\myindent Moreover for $u\in H_x^1$

\begin{equation}
    \begin{split}
    \|xS_Nu\|_{L^2}^2 &= \int_{\R^2} \frac{|x|^2}{N^2} \left|u\left(\frac{x}{N}\right)\right|^2 dx \\
    &= N^2 \|xu\|_{L^2}^2\\
    \|\nabla S_N u\|_{L^2}^2 &= \|\frac{1}{N} S_N\left(\nabla u\right)\|_{L^2}^2\\
    &= \frac{1}{N^2} \|\nabla u\|_{L^2}^2
    \end{split}
\end{equation}

\myindent We also modulate by multiplication by $e^{im|x|^2}$ for a given $m\in \R$ : this changes neither the $L^2$ norm neither $\|xu\|_{L^2}$. However for all $u \in H_x^1$ there holds

\begin{equation}
    \|\nabla\left(e^{i m |x|^2} u(x)\right)\|_{L^2}^2 = \left\|2im x u(x) + \nabla u(x)\right\|_{L^2}^2
\end{equation}

which yields 

\begin{equation}
    \begin{split}
    \text{When $u$ is real-valued}\  \|\nabla\left(e^{i m |x|^2} u(x)\right)\|_{L^2}^2 &= 4m^2 \|xu\|_{L^2}^2 + \|\nabla u\|_{L^2}^2\\
    \text{Otherwise}\  \|\nabla\left(e^{i m |x|^2} u(x)\right)\|_{L^2}^2 &\leq 2\left(4m^2 \|xu\|_{L^2} + \|\nabla u\|_{L^2}\right)
    \end{split}
\end{equation}

\myindent Combining these estimates yields the lemma

\begin{lemma}[energy estimates through modulation]\label{enest}
    When $u$ is real-valued there holds
    
    \begin{equation}
    \begin{split}
    \|S_N e^{im|y|^2} u(y)\|_{H_x^1}^2 &= \|xS_N e^{im|y|^2} u(y)\|_{L^2}^2 + \|\nabla S_N e^{im|y|^2} u(y)\|_{L^2}^2
    \\
    &= \left(N^2 + \frac{4m^2}{N^2}\right) \|xu\|_{L^2}^2 + \frac{1}{N^2} \|\nabla u\|_{L^2}^2
    \end{split}
    \end{equation}
    
    \myindent Otherwise 
    
    \begin{equation}
        \|S_N e^{im|y|^2} u(y)\|_{H_x^1}^2 \leq 2\left(N^2 + \frac{4m^2 + 1}{N^2}\right) \|u\|_{H_x^1}^2
    \end{equation}
\end{lemma}

\subsection{\textit{Modulated Equation}}
\myindent We remind equation \eqref{maineq}

\begin{equation}
    i\partial_t u(t,x) = -\Delta u(t,x) + |x|^2 u(t,x) + u(t,x)|u(t,x)|^2 + V(t,x) u(t,x)
\end{equation}

Modulation operators allow for the following change of function 

\begin{proposition}[modulated equation]\label{modeq}
    Let $L(t) > 0$, $\gamma(t)$, $b(t)$ be $\mathcal{C}^1$ functions over $\R_+$. Set 
    
    \begin{equation}
        u := e^{i\gamma(t)} S_{L(t)} w \qquad w(t,y) := e^{-i\frac{b(t)|y|^2}{4}} v(t,y) \qquad \frac{ds}{dt} = \frac{1}{L^2}
    \end{equation}
    
    \myindent Assume $t \to s(t)$ is invertible from $(t_0, + \infty)$ onto itself. Then $u(t,x)$ solves \eqref{maineq} if and only if $v(s,y)$ solves
    
    \begin{equation}\label{eqmod}
        i\partial_s v + \Delta v - \gamma_s v + \left(- L^4 + \frac{b_s}{4} - \frac{b^2}{4} - \frac{L_s}{L} \frac{b}{2}\right)|y|^2 v - i\left(\frac{L_s}{L} + b\right)(1 + \Lambda) v - v|v|^2 - W(s,y) v = 0
    \end{equation}
    
    where $L_s := \frac{d}{ds} (L(t(s)))$ and similar definitions for $b_s,\gamma_s$ ; where $\Lambda := y\cdot \nabla$ and finally where
    
    \begin{equation}
        V(t,x) = \frac{1}{L^2} W\left(s, \frac{x}{L}\right)
    \end{equation}
    
\end{proposition}
\begin{proof}

\myindent We compute

\begin{equation}
    \begin{split}
    i\partial_t u &= i\partial_t\left(e^{i\gamma(t)} \frac{1}{L(t)} w(t, \frac{x}{L(t)})\right) \\
    &= e^{i\gamma(t)}\left(-\gamma_t \frac{1}{L(t)} w(t,\frac{x}{L(t)}) - i\frac{L_t}{L^2} w + i\frac{1}{L} \partial_t w (t,\frac{x}{L(t)}) - i \frac{L_t}{L^3} x\cdot\nabla w (t,\frac{x}{L(t)})\right) \\
    &= e^{i\gamma} S_L \left(-i \frac{L_t}{L}(1 + \Lambda)w + i\partial_t w - \gamma_t w\right)\\
    &= \frac{e^{i\gamma}}{L^2} S_L\left(-i\frac{L_s}{L}(1 + \Lambda ) w  + i\partial_s w - \gamma_s w\right)
    \end{split}
\end{equation}

where we use only that $\partial_t = \frac{1}{L^2}\partial_s$. This ensures that $u$ solves \eqref{maineq} if and only if

\begin{equation}
    \frac{1}{L^2}\left(-i\frac{L_s}{L}(1 + \Lambda)w + i\partial_s w - \gamma_s w \right) = - S_L^{-1} \Delta S_L w + S_L^{-1} |x|^2 S_L w + S_L^{-1} \left(S_L w |S_Lw|^2\right) + \frac{1}{L^2} Ww
\end{equation}

\myindent We compute 

\begin{equation}
    \begin{split}
    \Delta\left(\frac{1}{L}u\left(\frac{x}{L}\right)\right) &= \frac{1}{L^3} (\Delta u)\left(\frac{x}{L}\right) = S_L\left(\frac{1}{L^2} \Delta u\right)\\
    |x|^2 \frac{1}{L} u\left(\frac{x}{L}\right) &= L^2 \frac{1}{L} \left| \frac{x}{L}\right|^2 u\left(\frac{x}{L}\right) = L^2 S_L\left(|y|^2 u(y)\right)\\
    \frac{1}{L^3} u\left(\frac{x}{L}\right)\left|u\left(\frac{x}{L}\right)\right|^2 &= \frac{1}{L^2} S_L(u|u|^2) 
    \end{split}
\end{equation}

\myindent Thus $u$ is a solution if and only if

\begin{equation}
    i\partial_s w - i\frac{L_s}{L}(1 + \Lambda) w - \gamma_s w = -\Delta w + L^4 |y|^2 w + w|w|^2 + Ww
\end{equation}

\myindent Let us here stress an essential algebraic fact : the cubic nonlinearity $u|u|^2$ is preserved through modulation, whereas should we have a nonlinearity of the form $u|u|^{p-1}$ a $L^{3 - p}$ would appear. Our proof doesn't therefore apply to other exponents. 

\myindent Using now that $w(s,y) = e^{-i\frac{b|y|^2}{4}}v(s,y)$ we compute

\begin{equation}
    \begin{split}
    i\partial_s w &= e^{-i\frac{b|y|^2}{4}}\left(\frac{b_s}{4} |y|^2v + i\partial_s v\right)\\
    \Lambda w &=e^{-i\frac{b|y|^2}{4}} y\cdot \left(-i\frac{b}{2} v y + \nabla v\right) = e^{-i\frac{b|y|^2}{4}}\left(-i \frac{b}{2}|y|^2 v + \Lambda v\right)\\
    \Delta w &=  \nabla \cdot e^{-i\frac{b|y|^2}{4}}\left(-i\frac{b}{2} v y + \nabla v\right) = e^{-i\frac{b|y|^2}{4}} \left(\Delta v -i\frac{b}{2} \Lambda v -ib v +\frac{b^2}{4}|y|^2 v -i \frac{b}{2} \Lambda v\right)
    \end{split}
\end{equation}

which yields the proposition
\end{proof}

\subsection{\textit{Hamiltonian Structure and Resonant Trajectory}}

\myindent The following result in proved in \cite{faou2020weakly} very thoroughly. We will not explain the proof here, which relies on carefully lead considerations on Hamiltonian dynamical systems.

\begin{proposition}[resonant trajectory]\label{restraj}
    Let 
    
    \begin{equation}
        \beta(s) := -\frac{\sin(4s)}{s \log s} \qquad s > 1
    \end{equation}
    
    \myindent There exists a $s_0 > 1$ and a solution with $\mathcal{C}^{\infty}$ regularity $(L,b)$ over $(s_0, + \infty)$ to the system 
    
    \begin{equation}
    \begin{cases} L^4 -\frac{b_s}{4} + \frac{b^2}{4} +  \frac{L_s}{L}\frac{b}{2} &= 1 + \beta(s) \\
    \frac{L_s}{L} + b &= 0 \end{cases}
    \end{equation}
    
    such that the associated Hamiltonian 
    
    \begin{equation}
        E(L,b) := \frac{1}{L^2}\left(\frac{b^2}{4} + 1 \right) + L^2 = \log s + O\left(\frac{\log s}{s}\right) \qquad s\to \infty
    \end{equation}
    
    \myindent Moreover there is a $B_0 > 0$ such that
    
    \begin{equation}
        \forall s\geq s_0 \qquad \frac{1}{B_0 \log s} \leq L^2 \leq B_0 \log s \qquad  |b(s)| \leq B_0(\log s)^3
    \end{equation}
    
    \myindent The time defined by $\frac{dt}{ds} = L^2 > 0$ and $t(s_0) = s_0$ over $(s_0, +\infty)$ satisfies 
    
    \begin{equation}
        |t(s) - s| \leq B_0(\log s)^2
    \end{equation}
    
    \myindent Thus, $t(s)$ is globally invertible. Finally, if we see $L(t),b(t)$ as functions of $t$, there holds
    
    \begin{equation}
        \left|\frac{d^k L}{dt^k} (t)\right| + \left|\frac{d^k b}{dt^k} (t)\right| + \left|\frac{d^k}{dt^k} \left(\frac{1}{L}\right)(t)\right| \leq B_k \left(\log t\right)^{\alpha_k}
    \end{equation}
    
\end{proposition}

\section{Soliton + perturbation decomposition and change of equation}

\myindent Putting together propositions \eqref{modeq} and \eqref{restraj}, the problem reduces to solving the equation

\begin{equation}\label{eqmod1}
    i\partial_s v = H v + \beta(s) |y|^2 v + \gamma_s v + v|v|^2 + W(s,y) v
\end{equation}

\myindent Let $\lambda = 2 + \eps$ with $\eps > 0$ small enough. Set $Q := Q_{\lambda}$ the corresponding soliton built in section III and set

\begin{equation}
    \gamma_s := -\lambda
\end{equation}

\myindent As $\beta(s)$ and $W(s,y)$ should decay to $0$ when $s\to \infty$, equation \eqref{eqmod1} roughly becomes when $s \to \infty$ 

\begin{equation}
    i\partial_s v = H v + v|v|^2 - \lambda v
\end{equation}

to which $v(s,y) = Q(y)$ is a stationary solution. We will therefore try and find a solution to \eqref{eqmod1} of the form 

\begin{equation}
    v(s,y) = Q(y) + w(s,y)
\end{equation}

where $w(s,y)$ is a perturbation decaying to zero when $s\to \infty$. Now, the interest of such a solution is that asymptotically $u$ will be equivalent to the bubble \eqref{ansatz1} whose $H_x^1$ norm is related to the square root of the energy $E(L(t),b(t))$ introduced in proposition \eqref{restraj} thus is grows as $(\log(t))^{\frac{1}{2}}$ when $t\to \infty$. 

\myindent This section and the following one are devoted to proving 

\begin{proposition}\label{princprop}
    Let $\beta(s) = -\frac{\sin(4s)}{s\log s}$. There exists constants $\alpha \in \R$, $\eps > 0$ and $s_0 > 1$ such that if we set 
    
    \begin{equation}
        W(s,y) := -\alpha \beta(s)Q_{2 + \eps}(y)
    \end{equation}
    
    then there exists a constant $B > 0$ and a solution $v \in \mathcal{C}^1((s_0,+\infty),H_x^3)$ such that
    
    \begin{equation}
        i\partial_sv = H v + \beta(s)|y|^2 v - \lambda v + v|v|^2 + W(s,y) v \qquad s\in[s_0, +\infty)
    \end{equation}
    
    where if
    
    \begin{equation}
        v(s,y) = Q(y) + w(s,y)
    \end{equation}
    
    there holds 
    
    \begin{equation} \label{boundw}
        \forall s\in[s_0, +\infty), \ \|w(s)\|_{H_x^{3}} \leq \frac{B}{s \log s}
    \end{equation}
\end{proposition}

\myindent We indicate that we cannot control only the $H_x^1$ norm of $w(s)$ as $H_x^1$ fails to be an algebra thus we could not control the cubic nonlinearity in the equation. Moreover the proof could be generalized to prove that the $H_x^{2k +1}$ norms of $w(s)$ all decay faster than $\frac{1}{s\log(s)}$ as $s \to \infty$. As indicated in the introduction, this would result in controlling the $H_x^k$ norms of the solution $u$ in theorem \eqref{main}. However, the proof would greatly increase in technical complexity without any real new idea, thus we have chosen to limit ourselves to the study of the $H_x^1$ norm.
 
\subsection{\textit{Change of function}}
\myindent As $HQ + Q^3 = \lambda Q$ (where we will denote $\lambda = 2+\eps$) we need to find $w$ a solution to 

\begin{equation}\label{eqw}
    i\partial_s w = (H-\lambda)w + 2Q^2 w + Q^2 \overline{w} + \beta(s)|y|^2 w + W(s,y) w + R(s) + N(s)
\end{equation}

where 

\begin{equation}
    \begin{split}
    R(s) &= \beta(s)|y|^2 Q(y) + W(s,y) Q(y)\\
    N(s) &= 2Q|w|^2 + Qw^2 + w|w|^2
    \end{split}
\end{equation}

\myindent The equation is no longer $\C$-linear : we thus set $w = \begin{pmatrix} w_1 \\ w_2\end{pmatrix}$ and write

\begin{equation}\label{eqw1}
\partial_s w = \mathcal{L} w + I\left(\beta(s)|y|^2 + W(s,y)\right) w + IR(s) + IN(s)
\end{equation}

where 

\begin{equation}
    \begin{split}
    \mathcal{L} &:= \begin{pmatrix} 0 & H_- \\ - H_+ & 0 \end{pmatrix} = \begin{pmatrix} 0 & H + Q^2 - \lambda \\ -(H + 3Q^2 - \lambda) & 0 \end{pmatrix}\\
    I &:= \begin{pmatrix} 0 & 1 \\ -1 & 0\end{pmatrix}\\
    R(s,y) &= \begin{pmatrix} \beta(s)|y|^2 Q(y) + W(s,y)Q(y) \\
    0\end{pmatrix}\\
    N(s,y) &= \begin{pmatrix} 2Q(y)|w(s,y)|^2 + Q(y)\left(w_1(s,y)^2 - w_2(s,y)^2\right) + w_1(s,y)|w(s,y)|^2\\
    2Q(y)w_1(s,y) w_2(s,y) + w_2(s,y)|w(s,y)|^2 \end{pmatrix}
    \end{split}
\end{equation}

\subsection{\textit{Spectral theory}}

\myindent We can write $H_{+/-} = H - 2 + A_{+/-}$ with $A_+ := 3Q^2 - \eps$ and $A_- := Q^2 - \eps$. Using lemma \eqref{infnorm} we find that taking $\eps > 0$ small enough yields that $A_{+/-}$ is a bounded operator over $L^2$ with operator norm arbitrarily small as $\eps \to 0$. Lemma \eqref{spec} yields that $H_{+/-}$ has a discreet spectrum with simple eigenvalues $\lambda_0^{+/-} < \lambda_1^{+/-} < ....$ and

\begin{equation}
    |\lambda_n^{+/-} - 4n| < \alpha
\end{equation}

uniformly in $n$, where $\alpha > 0$ is a given constant that we can choose arbitrarily small as $\eps \to 0$. We then observe that

\begin{equation}
    H_- Q = 0
\end{equation}

\myindent This follows from the definition of $Q$. Thus, $\lambda_0^- = 0$. Therefore, $H_-$ is a nonnegative self-adjoint operator with compact resolvent over $L^2(\R^2;\R)$ with kernel $\R Q$. As $H_+ = H_- + 2Q^2$ we infer that $\lambda_0^+ > 0$ which yields that $H_+$ is positive. Moreover we see straightforwardly that

\begin{lemma}\label{eqnorm+}
    Let $k \geq 0$. The norm defined by 
    
    \begin{equation}
        \|u\|_{H_+^k}^2 := \scal{H_+^k u}{u}
    \end{equation}
    
    is equivalent to the $H_x^k$ norms
\end{lemma}

\subsection{\textit{Preserved Energies}}

\myindent In order to estimate the $H_x^k$ norm of $w$, we wish to exhibit quantities preserved by the linearized flow $\exp(s\mathcal{L})$ of the equation.

\begin{lemma}[First preserved energy]
    Let $\mathcal{H} := \begin{pmatrix} H_+ & 0 \\ 0 & H_-\end{pmatrix}$. For all $u \in H_x^1$, set
    
    \begin{equation}
        E(u) := \frac{1}{2}\scal{\mathcal{H}u}{u} = \frac{1}{2}\scal{H_+ u_1}{u_1} + \frac{1}{2}\scal{H_- u_2}{u_2}
    \end{equation}
    
    \myindent Then the energy $E$ is preserved by the flow $\exp(s\mathcal{L})$
\end{lemma}

\begin{proof}
    It is enough to see that
    
    \begin{equation}
        \scal{\mathcal{H}\mathcal{L}u}{u} = 0 \qquad \forall u
    \end{equation}
\end{proof}

\myindent Now, as $H_{+/-} = H + A_{+/-}$ with $A_{+/-}$ a bounded operator over $L^2(\R^2;\R)$, we see that in order to bound the $H_x^1$ norm of $u$, it is enough to bound $E(u)$ and $\|u\|_{L^2}$. However, as $H_+$ is positive and as the kernel of $H_-$ is $\R Q$ we find that

\begin{equation}
    \exists c,C> 0, \forall u \in H^1_x \qquad c\|u\|_{H_x^1}^2 \leq E(u) + \scal{u_2}{Q}^2 \leq C\|u\|_{H_x^1}^2
\end{equation}

\myindent Let us now define 

\begin{equation}
    \rho := H_+^{-1} Q \in L^2(\R^2;\R)
\end{equation}

\myindent Up to a smaller choice of $\eps$, there holds 

\begin{equation}
    \scal{\rho}{Q} > 0
\end{equation}

\myindent Indeed, lemma \eqref{eqderiv} ensures that in fact $\rho = \partial_{\lambda} Q$ so that 

\begin{equation}
\scal{\rho}{Q} = \frac{1}{2}\partial_{\lambda} \|Q_{\lambda}\|_{L^2}^2 > 0
\end{equation}

\myindent The inequality follows from $\lambda \to \|Q_{\lambda}\|_{L^2}$ being increasing (lemma \eqref{normincr}) thus its derivative is almost everywhere positive.

\myindent This enables us to change slightly the previous estimate into

\begin{lemma}
    There exists $c,C> 0$, such that 
    
    \begin{equation}
        \forall u \in H^1_x \qquad c\|u\|_{H_x^1}^2 \leq E(u) + \scal{u_2}{\rho}^2 \leq C\|u\|_{H_x^1}^2
    \end{equation}
\end{lemma}

\myindent Now, $E(u)$ reveals itself insufficient to properly estimate $w$. Indeed, $H_x^1$ is not an algebra, thus we cannot control the nonlinearity using only $H_x^1$ norm. We therefore introduce another preserved energy

\begin{lemma}[Higher order preserved energy]
    Define $\mathcal{H}_3 := \begin{pmatrix} H_+ H_- H_+& 0 \\ 0 & H_-H_+H_-\end{pmatrix}$ and for $u\in H_x^3$ set
    
    \begin{equation}
        \mathcal{E}(u) := \frac{1}{2}\scal{\mathcal{H}u}{u} = \frac{1}{2}\scal{H_+ H_-H_+ u_1}{u_1} + \frac{1}{2}\scal{H_- H_+ H_- u_2}{u_2}
    \end{equation}
    
    \myindent The energy $\mathcal{E}$ is preserved by the flow $\exp(s\mathcal{L})$. Moreover, there exists $c,C > 0$ such that
    
    \begin{equation}\label{energy}
    \forall u \in H_x^3 \qquad c\scal{H^{3}u}{u} \leq \mathcal{E}(u) + E(u) + \scal{u_2}{\rho}^2 \leq C\scal{H^{3} u}{u}
\end{equation}
\end{lemma}

\begin{proof}
    That $\mathcal{E}$ is preserved is straightforward from the observation 
    
    \begin{equation}
        \scal{\mathcal{H}_3 \mathcal{L} u}{u} = 0 \qquad\forall u
    \end{equation}
    
    \myindent As for the energy estimates, the right inequality is straightforward, and for the left one we write for example that $H_+ H_- H_+ = H^3 + B$ where $B$ is a sum of terms of the form $A_1 A_2 A_3$ where $A_i$ is one of $H,Q^2 -\lambda,3Q^2 -\lambda$ with at most 2 of the $A_i$ equals to $H$. Thus we see that $|\scal{Bu}{u}|\leq C\|u\|_{H_x^2}^2$. Now, lemma \eqref{contnorm2} helps us bound this quantity by $\frac{1}{2}\|u\|_{H_x^3}^2 + C\|u\|_{L^2}^2$. Therefore

\begin{equation}
    \scal{H_+H_-H_+ u_1}{u_1} \geq \frac{1}{2}\|u_1\|_{H_x^3}^2 - C\|u_1\|_{L^2}^2
\end{equation}

and we know that $\|u_1\|_{L^2}$ can be bounded using $E(u)$ and $\scal{u_2}{\rho}$. We bound $\|u_2\|_{H_x^3}$ in the same way.

\end{proof}

\myindent Now, to bound the $H_x^3$ norm (which is an algebra norm) of a function $u$, one needs only bound $\mathcal{E}(u)$, $E(u)$, and $\scal{u_2}{\rho}^2$.

\subsection{\textit{Computation of $\exp(s\mathcal{L})$ and more spectral theory}}

\myindent We need to have an explicit expression for the linearized flow $\exp(s\mathcal{L})$. We first give a $\textit{formal}$ computation : we can heuristically compute

\begin{equation}
    \begin{split}
    &\mathcal{L}^{2n} =(-1)^n \begin{pmatrix} (H_-H_+)^n & 0 \\0 & (H_+H_-)^n\end{pmatrix}\\
    &\mathcal{L}^{2n+1} = (-1)^n\begin{pmatrix} 0 & H_-(H_+H_-)^n \\ -H_+(H_-H_+)^n & 0 \end{pmatrix}
    \end{split}
\end{equation}

\myindent Next, we see that $H_-$ and $H_+$ are nonnegative, so we can properly define over $L^2(\R^2;\R)$ the unbounded operator 

\begin{equation}
    A := \sqrt{H_+^{\frac{1}{2}} H_- H_+^{\frac{1}{2}}}
\end{equation}

\myindent Moreover, as $H_+$ is invertible, we compute

\begin{equation}
    \begin{split}
    \mathcal{L}^{2n} &=(-1)^n \begin{pmatrix}H_+^{-\frac{1}{2}} A^{2n} H_+^{\frac{1}{2}} & 0 \\ 0 & H_+^{\frac{1}{2}} A^{2n} H_+^{-\frac{1}{2}}\end{pmatrix}\\
    \mathcal{L}^{2n+1} &= (-1)^n\begin{pmatrix} 0 & H_+^{-\frac{1}{2}}A^{2n + 2} H_+^{-\frac{1}{2}}\\
    -H_+^{\frac{1}{2}} A^{2n} H_+^{\frac{1}{2}} &0 \end{pmatrix}
    \end{split}
\end{equation}

which yields

\begin{equation}\label{expformula}
    \exp(s\mathcal{L}) = \begin{pmatrix}H_+^{-\frac{1}{2}} \cos(s A)H_+^{\frac{1}{2}} & H_+^{-\frac{1}{2}} A\sin(sA) H_+^{-\frac{1}{2}} \\ 
-H_+^{\frac{1}{2}} s\text{sinc}(sA) H_+^{\frac{1}{2}} & H_+^{\frac{1}{2}} \cos(sA)H_+^{-\frac{1}{2}}\end{pmatrix}
\end{equation}

where we define $\text{sinc}(x) := \frac{\sin(x)}{x} = \sum_{n\geq 0} (-1)^n \frac{x^{2n}}{(2n+1)!}$ 

\myindent Now, this is only formal. However, if we define $\exp(s\mathcal{L})$ by \eqref{expformula}, we see that as $A$ is defined on $H_x^3$ then for all $u_0 \in H_x^3$ there holds as expected that $\exp(s\mathcal{L})u_0$ is the solution to $\partial_s v = \mathcal{L}v$ with initial condition $v(0) = u_0$.
 
\begin{lemma}[Spectrum of $A$]\label{specA}
    Let $\eta > 0$. There exists $\alpha > 0$ such that whenever $0 < \eps < \alpha$, $A$ is a nonnegative operator over $L^2$, self-adjoint, with compact resolvent. Moreover, its spectrum is formed of simple eigenvalues $0 \leq \mu_0 < \mu_1 < ...$ with 
    
    \begin{equation}
        \forall n\geq 0, \qquad |\mu_n - 4n| \leq \eta
    \end{equation}
    
   \myindent Finally, as $H_-$ has a nontrivial kernel, so has $A$. Therefore, $\mu_0 = 0$
\end{lemma}

\begin{proof}
    Take $\eta$ as in the statement, and $\delta > 0$ small, to be fixed later. Without loss of generality, we assume $\delta < 1$.
    
    \myindent We study $A^2 = H_+^{\frac{1}{2}} H_- H_+^{\frac{1}{2}} = H_+^2 - 2H_+^{\frac{1}{2}} Q^2 H_+^{\frac{1}{2}}$. First, let us study the resolvent of $A^2$ : we must ask ourselves, given a $\zeta \in \C$, when is $A^2 - \zeta$ invertible. To this effect, we introduce $0 < \lambda_0 < \lambda_1 < ...$ the spectrum of $H_+$, where $\lambda_n = 4n + \eps_n$ and we have that the $L^{\infty}$ norm of $(\eps_n)_n$ has limit zero when $\eps \to 0$, and we set $\phi_0, \phi_1,...$ an orthonormal basis of $L^2$ such that $H_+ \phi_n = \lambda_n \phi_n$. We set $D_0$ the disk of center $0$ and radius $\delta$, $D_1$ the disk of center $16$ and radius $\delta$, then $D_n$ the disk of center $16n^2$ and radius $n\delta$. We set $\Gamma_n$ the boundary circle of $D_n$ positively oriented. We finally set $\Sigma = \bigcup_{n\geq 0} D_n$. 
    
    \myindent Find $\zeta \in \C \backslash \Sigma$. We will show that, provided we take $\eps > 0$ small enough (depending on $\delta$ only), $A^2 - \zeta$ is invertible. First, the nth eigenvalue of $H_+^2$ is $16n^2 + 4n \eps_n + \eps_n^2$. Provided $\eps$ is small enough, we have for all $n\geq 1$ that
    
    \begin{equation}
        |16n^2 - \lambda_n^2|\leq \frac{1}{2} \delta n
    \end{equation}
    
    and for $n = 0$ that $|\lambda_0^2| \leq \frac{1}{2}\delta$. Thus, $H_+^2 - \zeta$ is invertible and we are able to formally write
    
    \begin{equation}\label{resolv}
        (A^2 - \zeta)^{-1} = (H_+^2 - \zeta)^{-1} \left(I - 2H_+^{\frac{1}{2}}Q^2 H_+^{\frac{1}{2}} (H_+^2 - \zeta)^{-1}\right)^{-1}
    \end{equation}
    
    where $I$ is the identity over $L^2$. Our goal is to prove that $2H_+^{\frac{1}{2}}Q^2 H_+^{\frac{1}{2}} (H_+^2 - \zeta)^{-1}$ is a bounded operator over $L^2$ with operator norm strictly smaller than $1$ which will prove the invertibility of $A^2 - \zeta$. To this effect, we first observe that there exists a universal constant $C > 0$ such that given any $v\in H_x^1$ there holds
    
    \begin{equation}
        \scal{H_+(Q^2 v)}{Q^2 v} \leq C\eps^2 \scal{H_+ v}{v}
    \end{equation}
    
    \myindent Indeed, we can write $H_+ = H + B$ where $B$ is a bounded operator over $L^2$, so 
    
    \begin{equation}
        \begin{split}
        \scal{H_+(Q^2 v)}{Q^2 v} &= \scal{H(Q^2 v)}{Q^2 v} + \scal{B(Q^2 v)}{Q^2 v}\\
        &\leq \|\nabla(Q^2 v)\|_{L^2}^2 + \|xQ^2 v\|_{L^2}^2 + \|B\|\|Q^2 v\|_{L^2}^2
        \end{split}
    \end{equation}
    
    \myindent Now, as $\nabla{Q^2 v} = Q^2 \nabla v + 2Q\nabla Q v$,  proposition \eqref{infnorm} enables us to conclude.
    
    \myindent Find $u = \sum_{n\geq 0}\alpha_n \phi_n \in L^2$. There holds
    
    \begin{equation}
        \begin{split}
        \left\|2H_+^{\frac{1}{2}}Q^2 H_+^{\frac{1}{2}} (H_+^2 - \zeta)^{-1} u\right\|_{L^2}^2 &\leq C \eps^2 \|H_+(H_+^2 - \zeta)^{-1} u\|_{L^2}^2\\
        &\leq C\eps^2 \sum_{n\geq 0} \frac{\lambda_n}{|\lambda_n^2 - \zeta|} |\alpha_n|^2
        \end{split}
    \end{equation}
    
    \myindent Given $n \geq 1$, as $\zeta$ is outside $D_n$, $|\lambda_n^2 - \zeta| \geq \frac{1}{2}n\delta$ and thus $\frac{\lambda_n}{|\lambda_n^2 - \zeta|} \leq \frac{4n + \eps_n}{\frac{1}{2} n\delta} \leq \frac{10}{\delta}$. Similarly, we have $\frac{\lambda_0}{|\lambda_0^2 - \zeta|} \leq \frac{10}{\delta}$. Thus, in terms of operator norm over $L^2$,
    
    \begin{equation}\label{specnormcont}
        \left\|2H_+^{\frac{1}{2}}Q^2 H_+^{\frac{1}{2}} (H_+^2 - \zeta)^{-1}\right\| \leq \frac{C'\eps}{\sqrt{\delta}}
    \end{equation}
    
    where $C'$ is a constant. Provided $\eps$ is small enough, the right-hand side is strictly smaller than $1$ and thus $I - 2H_+^{\frac{1}{2}}Q^2 H_+^{\frac{1}{2}} (H_+^2 - \zeta)^{-1}$ is invertible and formula \eqref{resolv} is correct.
    
    \myindent First, this ensures that $A^2$ has compact resolvent (indeed its resolvent is the product of a bounded operator over $L^2$ and of a compact one, thus is it compact). Therefore we already know that $A^2$ has a discreet spectrum and that there is an orthonormal basis of $L^2$ formed of eigenfunctions for $A^2$. Moreover, as we know that $A^2$ has no spectrum outside $\Sigma$, its eigenvalues are located within the $D_n$'s. We now prove that $A^2$ has exactly one simple eigenvalue in $D_n$. Indeed, following Kato \cite[p.178]{kato2013perturbation} we know that the spectral projector over the eigenfunctions corresponding to the eigenvalues located within $D_n$ can be written as 
    
    \begin{equation}
        \pi_n^{\eps} = \frac{1}{2i\pi} \int_{\Gamma_n} (A^2 - \zeta)^{-1} d\zeta
    \end{equation}
    
    \myindent However, using equation \eqref{specnormcont}, we see that 
    
    \begin{equation}
        \left\| \left(I - 2H_+^{\frac{1}{2}}Q^2 H_+^{\frac{1}{2}} (H_+^2 - \zeta)^{-1}\right)^{-1} - I\right\| \leq \frac{C\eps}{\sqrt{\delta}}
    \end{equation}
    
    for a constant $C$ when $\eps$ is small enough, which is an immediate consequence of the formula 
    
    \begin{equation}
        (I- B)^{-1} - I = \sum_{n\geq 1} B^n
    \end{equation}
    
    whenever $\|B\| < 1$. Thus, if we write 
    
    \begin{equation}
        \pi_n := \frac{1}{2i\pi} \int_{\Gamma_n} (H_+^2 - \zeta)^{-1} d\zeta
    \end{equation}
    
    there holds
    
    \begin{equation}
        \begin{split}
        \left\|\pi_n^{\eps} - \pi_n \right\| &= \frac{1}{2\pi} \left\|\int_{\Gamma_n} (H_+^2 - \zeta)^{-1} \left( \left(I - 2H_+^{\frac{1}{2}}Q^2 H_+^{\frac{1}{2}} (H_+^2 - \zeta)^{-1}\right)^{-1} - I\right)d\zeta\right\| \\
        &\leq \frac{C\eps}{\sqrt{\delta}} \frac{1}{2\pi} \int_{\Gamma_n} \left\|(H_+^2 - \zeta)^{-1}\right\| d\zeta\\
        & \leq \frac{C\eps}{\sqrt{\delta}} \frac{1}{2\pi}\int_{\Gamma_n} \frac{1}{\frac{1}{2} n\delta} d\zeta\\
        &\leq C \frac{\eps}{\sqrt{\delta}}
        \end{split}
    \end{equation}
    
    \myindent For a constant $C$ which is independent of $n$ and of $\eps$. This upper bound can be made arbitrarily small when $\eps \to 0$ : should it be smaller than $1$ strictly, this ensures that $\pi_n^{\eps}$ and $\pi_n$ have the same range (see \cite[~Lemma 4.1 p.34]{kato2013perturbation}) which is equal to $1$ as $H_+$ has simple eigenvalues.
    
    \myindent Therefore, provided $\eps > 0$ is small enough, and this depending only of $\delta$, we have proved that $A^2$ has a discreet spectrum formed of simple eigenvalues $0 \leq \mu_0^2 < \mu_1^2 < ....$ with $\mu_n^2 = 16n^2 + \alpha_n$, where $|\alpha_n|\leq \delta max(1,n)$ (indeed, $\mu_n^2 \in D_n$).
    
    \myindent Thus, $A$ has the discreet spectrum formed of simple eigenvalues $0 \leq \mu_0 <\mu_1 < ...$ and we see that $|\mu_0|\leq \sqrt{\delta}$, and moreover that for $n\geq 1$
    
    \begin{equation}
        \begin{split}
        |\mu_n - 4n| &= \left|\sqrt{16 n^2 + \alpha_n} - 4n\right|\\
        &= 4n \left|\sqrt{1 + \frac{\alpha_n}{16n^2}} - 1\right|\\
        &\leq 4n C \frac{\alpha_n}{16n^2}\\
        &\leq C \delta
        \end{split}
    \end{equation}
    
    where $C$ is a constant. Thus, provided we initially choose 
    
    \begin{equation}
        \max(\sqrt{\delta}, C\delta) < \eta
    \end{equation}
    
    we have proved the statement.
\end{proof}

\myindent We fix until the end $\psi_0, \psi_1,...$ an $L^2$ orthonormal basis of eigenfunctions for $A$.

\section{Backward Integration}

\myindent In order to prove proposition \eqref{princprop}, we will use a $\textit{backward integration}$ scheme. The idea is to use an advanced bootstrap argument as follows. Using the a priori bounds on $V$, we see that the equation \eqref{eqw} is locally well-posed in $H_x^3$, thus we have existence and uniqueness of local solutions for a given initial condition. Thus, given $0 < s_0 < M$, we may set $w^M$ the unique solution to \eqref{eqw} defined close to $M$ such that $w^M(M) = 0$. We will prove that for a choice of $s_0, B > 0$ big enough, $w^M$ is defined over $(s_0,M]$ for all $M > s_0$ and there holds

\begin{equation}
    \|w^M(s)\|_{H_x^3} \leq \frac{B}{s \log s} \qquad s_0 < s < M
\end{equation}

\myindent Moreover, we will prove that $(w^M)$ is a Cauchy sequence in $\mathcal{C}((s_0,T],H_x^3)$ for all $T > s_0$. If $w$ is its limit, it solves the integral form of equation \eqref{eqw}. Thus, it is a strong solution in $\mathcal{C}^1((s_0,+\infty),H_x^3)$ and the bound \eqref{boundw} holds.

\subsection{\textit{Another change of function}}
\myindent Extracting the solution to the linearized equation, set

\begin{equation}
    f^M := w^M + \int_s^M \exp((s-\sigma)\mathcal{L}) IR(\sigma)d\sigma := w^M + r^M
\end{equation}

\myindent $f^M$ is a solution to

\begin{equation}\label{eqf}
    \partial_s f^M = \mathcal{L}f^M + IK(s)f^M - IK(s)r^M(s) + IN(s)
\end{equation}

where we set $K(s) := \beta(s)(|y|^2 - \alpha Q(y))$. In order to estimate the $H_x^3$ norm of $w^M$, we need to estimate $f^M$ and $r^M$.

\subsection{\textit{Estimate of $r^M$}}

\myindent The estimate of $\|r^M(s)\|_{H_x^3}$ is a consequence of the \textit{oscillatory nature} of $r^M$. Indeed, as we will see, one can explicitly write $r^M(s)$ as a sum of integrals of the form

\begin{equation}
    \int_s^M \text{oscillatory function}\times \frac{d\sigma}{\sigma \log(\sigma)}
\end{equation}

Using somehow fastidious integration by parts, such integrals can always be bounded by $\frac{C}{s\log s}$.

\myindent Now, it will appear in the proof that the eigenvalue $\mu_1$ of $A$ may cause cancellation in the oscillations, thus forbidding to conclude the estimates on $r^M$. To prevent that, we here choose the value of $\alpha$ : we set 

\begin{equation}
    \alpha := \frac{\scal{|y|^2Q}{H_+^{-\frac{1}{2}}\psi_1}}{\scal{Q^2}{H_+^{-\frac{1}{2}}\psi_1}}
\end{equation}

such that $R(s) = \beta(s)(|y|^2 Q - \alpha Q^2)$ is orthogonal to $H_+^{-\frac{1}{2}} \psi_1$, or to say it differently that $H_+^{-\frac{1}{2}} R(s)$ is orthogonal to $\psi_1$. We need to justify that $\alpha$ is well defined, that is we need to prove that $\scal{Q^2}{H_+^{-\frac{1}{2}}\psi_1} \neq 0$ provided $\eps > 0$ is small enough. We compute to this effect that $H_+^{-\frac{1}{2}}\psi_1 \to \frac{1}{4} h_1$ in $L^2$ when $\eps \to 0$ (this follows from the computations of the proof of lemma \eqref{specA}). Moreover using \eqref{linapprox} we see that $Q^2 \sim C\eps h_0^2$ when $\eps\to 0$ in $L^2$ with $C > 0$ Thus, it is enough to know that the scalar product $\scal{h_0^2}{h_1}$ is nonzero, which is proved in \cite{faou2020weakly}.

\myindent We prove in this subsection the following lemma

\begin{lemma}\label{normr}
    There is a constant $B_0$ independent of $s$ and $M$ such that 
    
    \begin{equation}
        \|r^M(s)\|_{H_x^3} \leq \frac{B_0}{s\log s}
    \end{equation}
    
\end{lemma}

\myindent We start by writing

\begin{equation}
    \begin{split}
    r^M(s) &= -\int_s^M \frac{\sin(4\sigma)}{\sigma\log \sigma} \exp((s-\sigma)\mathcal{L}) \begin{pmatrix}0\\ -(|y|^2 Q - \alpha Q^2)\end{pmatrix} d\sigma\\
    &= \begin{pmatrix} H_+^{-\frac{1}{2}} A \int_s^M \frac{\sin(4\sigma)}{\sigma \log \sigma } \sin((s-\sigma) A) H_+^{-\frac{1}{2}}(|y|^2Q - \alpha Q^2) d\sigma)\\ H_+^{\frac{1}{2}} \int_s^M \frac{\sin(4\sigma)}{\sigma\log \sigma} \cos((s-\sigma)A) H_+^{-\frac{1}{2}}(|y|^2 Q - \alpha Q^2) d\sigma \end{pmatrix}
    \end{split}
\end{equation}

\myindent However 

\begin{multline}
    \sin(4\sigma)\sin((s-\sigma)A) H_+^{-\frac{1}{2}}(|y|^2 Q - \alpha Q^2) =\\ \frac{1}{2}\left(\cos(sA - \sigma(A + 4)) - \cos(sA - \sigma(A - 4))\right)H_+^{-\frac{1}{2}}(|y|^2 Q - \alpha Q^2)
\end{multline}

\myindent And we have that 

\begin{equation}
    \cos(sA - \sigma(A-4)) H_+^{-\frac{1}{2}}(|y|^2 Q - \alpha Q^2) = \sum_{n\neq 1}\cos(s\mu_n - \sigma(\mu_n - 4))\scal{H_+^{-\frac{1}{2}}(|y|^2 Q - \alpha Q^2)}{\psi_n} \psi_n
\end{equation}

where we are able to exclude $n=1$ because of the definition of $\alpha$. Here, we insist on the need to exclude $\alpha = 1$ : it could be that $\mu_1 = 4$ and in that case there would be a term in the sum which doesn't oscillate with respect to $\sigma$.

\myindent Now, if we denote $\phi(y) := H_+^{-\frac{1}{2}} (|y|^2 Q - \alpha Q^2)$ there holds 

\begin{equation}
    \int_s^M \frac{\cos(sA - \sigma(A-4))}{\sigma \log\sigma} \phi(y)d\sigma = \sum_{n\neq 1}\left(\int_s^M \frac{\cos(s\mu_n - \sigma(\mu_n - 4))}{\sigma \log \sigma} d\sigma\right) \scal{\phi(y)}{\psi_n}\psi_n
\end{equation}

\myindent However

\begin{multline}
    \int_s^M \frac{\cos(s\mu_n - \sigma(\mu_n - 4))}{\sigma \log \sigma } d\sigma = \left[-\frac{\sin(s\mu_n - \sigma(\mu_n - 4))}{(\mu_n - 4)\sigma \log \sigma }\right]_s^M\\ - \int_s^M \frac{\sin(s\mu_n - \sigma(\mu_n - 4))}{\mu_n - 4} \frac{\log \sigma + 1}{\sigma^2 (\log \sigma)^2} d\sigma
\end{multline}

\myindent As $n\neq 1$ we see that $|\mu_n - 4|\geq c > 0$ with $c$ a constant thanks to \eqref{specA}. Thus, we have the bound

\begin{equation}
    \left|\int_s^M \frac{\cos(s\mu_n - \sigma(\mu_n - 4))}{\sigma \log \sigma } d\sigma\right| \leq \frac{B_0}{s\log s}
\end{equation}

where $B_0$ does not depend on $M$ nor on $n$.

\myindent The same kind of reasoning gives

\begin{equation}
    \int_s^M \frac{\cos(sA - \sigma(A+4))}{\sigma \log\sigma} \phi(y)d\sigma = \sum_{n\neq 1}\left(\int_s^M \frac{\cos(s\mu_n - \sigma(\mu_n + 4))}{\sigma \log \sigma } d\sigma\right) \scal{\phi(y)}{\psi_n}\psi_n
\end{equation}

where 

\begin{equation}
    \left|\int_s^M \frac{\cos(s\mu_n - \sigma(\mu_n + 4))}{\sigma \log \sigma } d\sigma\right| \leq \frac{B_0}{s\log s}
\end{equation}

\myindent Let us now see that writing 

\begin{equation}
    r_1^M(s) = H_+^{-\frac{1}{2}} A \sum_{n\neq1} \alpha_n(s) \scal{\phi(y)}{\psi_n}\psi_n
\end{equation} 

with $|\alpha_n(s)|\leq \frac{B_0}{s\log s}$ yields

\begin{equation}
    \begin{split}
    \|r_1^M(s)\|_{H_x^3}^2 &\leq C\|r_1^M(s)\|_{H_+^3}^2\\
    &= C\left\|H_+^{-\frac{1}{2}} A \sum_{n\neq 1} \alpha_n(s) \scal{\phi(y)}{\psi_n} \psi_n\right\|_{H_+^3}^2\\
    &= C\left\|A\sum_{n\neq 1}\alpha_n(s)\scal{\phi(y)}{\psi_n} \psi_n\right\|_{H_+^2}^2
    \end{split}
\end{equation}

\myindent Now, given a $u\in H_x^2$, there holds

\begin{equation}
    \begin{split}
    \|u\|_{H_+^2}^2 &= \scal{H_+^2 u}{u}\\
    &= \scal{A^2 u}{u} + 2\scal{H_+^{\frac{1}{2}} Q^2 H_+^{\frac{1}{2}} u}{u}\\
    &\leq \scal{A^2u}{u} + \eps \|u\|_{H_+^1}^2\\
    &\leq \scal{A^2 u}{u} + \frac{1}{2} \|u\|_{H_+^2}^2 + C\|u\|_{L^2}^2
    \end{split}
\end{equation}

where $C$ is a constant independent of $u$. Thus, we see that 

\begin{equation}
    \|u\|_{H_+^2}^2 \leq C\left(\scal{A^2 u}{u} + \|u\|_{L^2}^2\right)
\end{equation}

from which follows

\begin{equation}
    \begin{split}
    \|r_1^M(s)\|_{H_x^3}^2 &\leq C\scal{A^2 A \sum_{n\neq 1}\alpha_n(s) \scal{\phi(y)}{\psi_n} \psi_n}{A\sum_{n\neq 1} \alpha_n(s) \scal{\phi(y)}{\psi_n} \psi_n} + C\left\|A\sum_{n\neq 1} \alpha_n(s) \scal{\phi(y)}{\psi_n} \psi_n\right\|_{L^2}^2\\
    &= \sum_{n\neq 1} (\mu_n^4 + \mu_n^2) \alpha_n(s)^2 \left|\scal{\phi(y)}{\psi_n} \right|^2\\
    &\leq \frac{B_0^2}{s^2 (\log s)^2} \sum_{n\neq 1} (\mu_n^4 + \mu_n^2) \left|\scal{\phi(y)}{\psi_n}\right|^2\\
    &\leq \frac{B_0^2}{s^2 (\log s)^2} C \|\phi(y)\|_{H_x^4}^2
    \end{split}
\end{equation}

wich we can write as

\begin{equation}
    \|r_1^M(s)\|_{H_x^3} \leq \frac{B_0}{s \log s}
\end{equation}

with $B_0$ another constant without importance.

\myindent Now, we similarly show that 

\begin{equation}
    r_2^M(s) = H_+^{\frac{1}{2}} \sum_{n\neq1} \beta_n(s) \scal{\phi(y)}{\psi_n}\psi_n
\end{equation} 

with $|\beta_n(s)|\leq \frac{B_0}{s\log s}$ uniformly in $n$. Thus, we are able to write 

\begin{equation}
    \begin{split}
    \|r_2^M(s)\|_{H_x^3}^2 &\leq C\|r_2^M(s)\|_{H_+^3}^2 \\
    &= \left\|\sum_{n\neq 1} \beta_n(s) \scal{\phi(y)}{\psi_n}\psi_n \right\|_{H_+^4}^2
    \end{split}
\end{equation}

\myindent A few computations show that 

\begin{equation}
    \|u\|_{H_+^4}^2 \leq C\left(\scal{A^4 u}{u} + \|u\|_{L^2}^2\right)
\end{equation}

\myindent Thus, there holds 

\begin{multline}
    \|r_2^M(s)\|_{H_x^3}^2 \leq C\scal{A^4 \sum_{n\neq 1} \beta_n(s) \scal{\phi(y)}{\psi_n}\psi_n}{\sum_{n\neq 1}\beta_n(s) \scal{\phi(y)}{\psi_n}\psi_n}\\ + C\left\|\sum_{n\neq 1}\beta_n(s) \scal{\phi(y)}{\psi_n}\psi_n\right\|_{L^2}^2
\end{multline}

and we can conclude in the same way that 

\begin{equation}
    \|r_2^M(s)\|_{H_x^3} \leq \frac{B_0}{s\log s}
\end{equation}

\subsection{\textit{An estimate}}

\myindent We will need the following technical result

\begin{lemma}\label{intr}
    There is a constant $B_0$ independent of $M$ such that
    
    \begin{equation}
        \left|\int_s^M \scal{r_1^M(\sigma)}{Q}\right| \leq \frac{B_0}{s\log s}
    \end{equation}
\end{lemma}

\begin{proof} 
This lemma is essentially a consequence of the oscillatory nature of $r_1^M$ on the one hand, and on the other hand of the particular algebraic fact that $\scal{H_+^{\frac{1}{2}}r_1^M}{\psi_i} = 0$ for $i = 0,1$. Indeed, we set 

\begin{equation}
    H_+^{-\frac{1}{2}}(|y|^2Q - \alpha Q^2) = \sum_{n\neq 1} \alpha_n \psi_n
\end{equation}

thus, 

\begin{equation}
    \begin{split}
    r_1^M(s) &= H_+^{-\frac{1}{2}} A \int_s^M \frac{\sin(4\sigma)}{\sigma \log \sigma} \sin((s-\sigma)A) H_+^{-\frac{1}{2}} (|y|^2Q - \alpha Q^2)d\sigma \\
    &=\frac{1}{2} H_+^{-\frac{1}{2}} A \sum_{n\geq 2} \alpha_n\left(\int_s^M \frac{\cos(s\mu_n - \sigma(\mu_n+4)) - \cos(s\mu_n - \sigma(\mu_n - 4))}{\sigma \log \sigma} d\sigma \right) \psi_n
    \end{split}
\end{equation}

where we can remove both $n =0$ and $n =1$ because there is a factor $A$ before the sum. Now,

\begin{multline}
    \int_s^M \scal{r_1^M(\sigma)}{Q} = \\ \scal{\sum_{n\geq 2} \left(\int_s^M \left(\int_{\sigma}^M \frac{\cos(\sigma\mu_n - \tau(\mu_n+4)) - \cos(\sigma\mu_n - \tau(\mu_n - 4))}{\tau \log \tau} d\tau\right)d\sigma \right)\alpha_n \psi_n}{AH_+^{-\frac{1}{2}} Q}
\end{multline}

\myindent We then compute

\begin{multline}
    \int_{\sigma}^M \frac{\cos(\sigma\mu_n - \tau(\mu_n-4))}{\tau \log \tau} d\tau = \left[-\frac{\sin(\sigma\mu_n - \tau(\mu_n-4))}{(\mu_n - 4) \tau\log \tau}\right]_{\sigma}^M \\
    - \int_{\sigma}^M \frac{\sin(\sigma\mu_n - \tau(\mu_n-4))}{(\mu_n - 4)} \frac{\log\tau + 1}{\tau^2 (\log \tau)^2} d\tau
\end{multline}

\myindent Thus

\begin{multline}
    \int_{\sigma}^M \frac{\cos(\sigma\mu_n - \tau(\mu_n-4))}{\tau \log \tau} d\tau = -\frac{\sin(\sigma\mu_n - M(\mu_n - 4))}{(\mu_n - 4)M\log M} + \frac{\sin(4\sigma)}{(\mu_n-4) \sigma\log \sigma}
    \\
    - \int_{\sigma}^M \frac{\sin(\sigma\mu_n - \tau(\mu_n-4))}{(\mu_n - 4)} \frac{\log\tau + 1}{\tau^2 (\log \tau)^2} d\tau
\end{multline}

\myindent First, there holds

\begin{equation}
    \begin{split}
    \int_s^M \frac{\sin(\sigma\mu_n - M(\mu_n - 4))}{(\mu_n - 4)M\log M}d\sigma &= -\frac{\cos(4M)}{\mu_n(\mu_n-4) M\log M} + \frac{\cos(s\mu_n - M(\mu_n-4))}{\mu_n(\mu_n-4) M\log M}\\
    \int_s^M \frac{\sin(4\sigma)}{(\mu_n-4) \sigma\log \sigma} d\sigma &= \left[-\frac{\cos(4\sigma)}{4(\mu_n-4)\sigma\log \sigma}\right]_s^M - \int_s^M \frac{\cos(4\sigma)}{4(\mu_n - 4)} \frac{1 + \log \sigma}{\sigma^2(\log\sigma)^2} d\sigma
    \end{split}
\end{equation}

with all the right-hand terms bounded by $\frac{B_0}{s\log s}$ (because $n\geq 2$)

\myindent Moreover,

\begin{multline}
    \int_{\sigma}^M \frac{\sin(\sigma\mu_n - \tau\mu_n-4))}{(\mu_n - 4)} \frac{\log\tau + 1}{\tau^2 (\log \tau)^2} d\tau = \left[-\frac{\cos(\sigma\mu_n - \tau(\mu_n - 4))}{(\mu_n - 4)^2}\frac{\log \tau +  1}{\tau^2(\log \tau)^2}\right]_{\sigma}^M\\
    - \int_{\sigma}^M \frac{\cos(\sigma\mu_n - \tau(\mu_n - 4))}{(\mu_n - 4)^2}\left(\frac{1}{\tau^3(\log\tau)^2} - (1+\log\tau)\frac{2\tau(\log\tau)^2 + 2\tau\log \tau }{\tau^4 (\log \tau)^4}\right)d\tau
\end{multline}

the bracket and the integral are both bounded by $\frac{B_0}{s^2 \log s}$. Using now that 

\begin{equation}
    \int_s^M \frac{d\sigma}{\sigma^2 \log \sigma} \leq \frac{C}{s \log s}
\end{equation}

with $C$ independent of $M$, we are able to conclude.

\myindent The computations are exactly the same for the other term $\int_s^M \frac{\cos(s\mu_n - \sigma(\mu_n+4))}{\sigma \log \sigma} d\sigma$
\end{proof}

\subsection{\textit{Estimate of $f^M$}}

\myindent We now want to bound the $H_x^3$ norm of $f^M$. Contrary to the estimate of $r^M$ we don't have a priori oscillatory behaviour or explicit form. Thus, we will use a \textit{bootstrap argument}. \eqref{energy} tells us that we need only bound $\mathcal{E}(f^M)$, $E(f^M)$, and $\scal{f_2^M}{\rho}^2$. We thus set given $B> 0$

\begin{equation}
    T_{[s_0, M]}(B) := \inf\left\{s\in [s_0, M], \forall s < \sigma < M, \mathcal{E}(f^M(\sigma)) + E(f^M(\sigma)) + \scal{f_2^M(\sigma)}{\rho}^2 \leq \frac{B^2}{\sigma^2 (\log \sigma)^2}\right\}
\end{equation}

\myindent Let us first remark that as there holds local existence and uniqueness in $H_x^3$ then the existence of an a priori bound on $[s,M]$ yields that $w^M$ is well-defined on an open neighborhood of $[s,M]$. Thus, $T_{[s_0,M]}(B)$ is well-defined. 

\myindent We prove in this subsection

\begin{lemma}\label{boundf}
    There exists $s_0>0$, $B > 0$ such that for all $M \geq s_0$ there holds
    \begin{equation}
        T_{[s_0,M]}(B) = s_0
    \end{equation}
    \myindent In particular, for all $M\geq s_0$, $f^M$ is defined on $[s_0, M]$.
    
    \myindent Finally, there exists a $C > 0$ such that for all $M\geq s_0$ there holds
    \begin{equation}
        \forall s_0 \leq s \leq M \ \|f^M(s)\|_{H_x^3} \leq \frac{CB}{s \log(s)} 
    \end{equation}
\end{lemma}

\myindent Thanks to local existence and uniqueness, to prove the proposition, it suffices to prove that given $s_0, B$ large enough there holds for all $M > s_0$

\begin{equation}
    \forall\ T_{[s_0,M]}(B) \leq s \leq M \quad \mathcal{E}(f^M(\sigma)) + E(f^M(\sigma)) + \scal{f_2^M(\sigma)}{\rho}^2 \leq \frac{1}{2}\frac{B^2}{\sigma^2 (\log \sigma)^2}
\end{equation}

\myindent Indeed, there cannot hold $T_{[s_0,M]}(B) > s_0$ in that situation.

\quad

\myindent Now, the interest of a bootstrap argument is that we can  work only on $[T_{[s_0,M]}(B),M]$, where we already know a priori estimates. Let us take $T_{[s_0, M]}(B) < s < M$. There holds thanks to \eqref{eqf}

\begin{equation}
    \partial_s \begin{pmatrix} f_1^M\\f_2^M\end{pmatrix} = \begin{pmatrix} H_- f_2^M + \beta(s)(|y|^2 - \alpha Q(y))f_2^M - \beta(s)(|y|^2 - \alpha Q(y)) r_2^M\\ -H_+ f_1^M - \beta(s)(|y|^2 - \alpha Q(y))f_1^M + \beta(s)(|y|^2 - \alpha Q(y)) r_1^M\end{pmatrix} + IN(s)
\end{equation}

which yields by definition of $\rho = H_+^{-1}Q$

\begin{equation}
    \begin{split}
    \partial_s \scal{f_2^M}{\rho} &= -\scal{f_1^M}{Q} - \beta(s)\scal{f_1^M}{|y|^2\rho - \alpha Q\rho} + \beta(s)\scal{r_1^M(s)}{|y|^2 \rho - \alpha Q\rho} - \scal{N(s)_1}{\rho} \\
    \partial_s E(f) &= \scal{\mathcal{H}IK(s)f}{f} - \scal{\mathcal{H}IK(s)r^M}{f} + \scal{\mathcal{H} f}{IN(s)}\\
    \partial_s \mathcal{E}(f) &= -\scal{H_+H_-H_+ K(s)f_2}{f_1} + \scal{H_-H_+H_- K(s) f_1}{f_2} - \scal{\mathcal{H}_3 IK(s)r^M(s)}{f} + \scal{\mathcal{H}_3f}{IN(s)}
    \end{split}
\end{equation}

\subsubsection{\textit{\textmd{Estimate of $\scal{f^M_2}{\rho}^2$}}}

\myindent Let us study the first equation. The following algebraic fact is fundamental (and it is here that it is useful to have introduced $\rho$) : $v^M := w^M + Q$ solves 
\begin{equation}
\begin{cases}
i\partial_s v^M &= (H-\lambda) v^M + \beta(s)(|y|^2- \alpha Q)v^M + v^M|v^M|^2 \\
    v^M(M) &= Q
\end{cases}
\end{equation}

\myindent This equation preserves the $L^2$ norm. We thus see that 

\begin{equation}
    \|w^M(s)\|_{L^2}^2 + 2\scal{w^M_1(s)}{Q} + \|Q\|_{L^2}^2 = \|w^M(s) + Q\|_{L^2}^2 = \|Q\|_{L^2}^2
\end{equation}

\myindent Thus, $\scal{w_1^M}{Q}$ is \textit{quadratic} rather than linear as could be expected : 

\begin{equation}
    |\scal{w_1^M(s)}{Q}| = \frac{1}{2} \|w^M(s)\|_{L^2}^2 \leq \frac{C B^2}{s^2 (\log s)^2}
\end{equation}

as long as $T_{[s_0, M]}(B) < s < M$. Now, using $f_1^M = w_1^M - r_1^M$, there holds

\begin{multline}
    \scal{f_2^M(s)}{\rho} = -\int_s^M \scal{r_1^M(\sigma)}{Q} d\sigma \\
    + \int_s^M \left(\scal{w_1^M(\sigma)}{Q} + \beta(\sigma)\scal{f_1^M(\sigma)}{|y|^2 \rho - \alpha Q\rho} - \beta(\sigma)\scal{r_1^M(\sigma)}{|y|^2\rho-\alpha Q\rho} + \scal{N(\sigma)_1}{\rho}\right)d\sigma
\end{multline}

\myindent Using lemma \eqref{intr}, the first integral is bounded by $\frac{B_0}{s\log s}$ with $B_0$ universal, and the second integral's integrand can be bounded by $\frac{CB^2}{\sigma^2(\log \sigma)^2}$ with $C$ universal, thus we can bound the second integral by $\frac{CB^2}{s(\log s)^2}$ with $C$ a constant. We finally get

\begin{equation}
    \scal{f_2^M(s)}{\rho}^2\leq \frac{CB^4}{s^2(\log s)^4} + \frac{B_0^2}{s^2 (\log s)^2} \leq \frac{1}{6} \frac{B^2}{s^2 (\log s)^2}
\end{equation}

provided we choose $B$ big enough with regard to $B_0$ then $s_0$ big enough.

\subsubsection{\textit{\textmd{Estimate of $E(f^M)$}}} 

\myindent We first recall lemma \eqref{xcarre} 

\begin{equation}
    \||y|^2u\|_{H_x^r} \leq C_r \|u\|_{H_x^{r+1}}
\end{equation}

\myindent This enables us to estimate via the second equation and \eqref{normr}

\begin{equation}
    |\partial_s E(f)|\leq \frac{C}{s\log s} \|f\|_{H_x^3}^2 + \frac{C}{s^2(\log s)^2}\|f\|_{H_x^3} + \|f\|_{H_x^3} \|N(s)\|_{L^2}
\end{equation}

\myindent Moreover there holds $\|N(s)\|_{L^2} \leq \|N(s)\|_{H_x^3}$. Now, $H_x^3$ is an algebra thanks to \eqref{algebra} and $N(s)$ is quadratic in $w^M$, so we have the bound 

\begin{equation}
    \|N(s)\|_{L^2} \leq C \|w^M\|_{H_x^3}^2
\end{equation}

\myindent Using now $w^M = f^M - r^M$ and the hypothesis $\|f^M\|_{H_x^3} \leq \frac{B}{s\log s}$ and \eqref{normr} we have by choosing $B \geq B_0$ that 

\begin{equation}
    \|N(s)\|_{L^2} \leq \frac{C B^2}{s^2 (\log s)^2}
\end{equation}

\myindent This ensures that 

\begin{equation}
    |\partial_s E(f^M)| \leq \frac{CB^3}{s^3 (\log s)^3}
\end{equation}

\myindent Integrating between $s$ and $M$ (as $f^M(M) = 0$) yields

\begin{equation}
    E(f^M) \leq \frac{C B^3}{s^2 (\log s)^3} \leq \frac{1}{6} \frac{B^2}{s^2 (\log s)^2}
\end{equation}

as long as $CB \leq \frac{1}{6} \log s_0$

\subsubsection{\textit{\textmd{Estimate of $\mathcal{E}(f^M)$}}} 

\myindent We finally need to bound $\mathcal{E}(f^M)$. Examining the right-hand side of the third equation, we can decompose it as $\beta(s)\left(-\scal{H^3 |y|^2 f_2}{f_1} + \scal{H^3 |y|^2 f_1}{f_2}\right) +$ (remainder) where the remainder can be bounded using that $H_x^3$ is an algebra (\eqref{algebra}), lemma \eqref{xcarre}, lemma \eqref{normr} and proposition \eqref{infnorm}, and the a priori bound over $\|f^M\|_{H_x^3}$, by $\frac{CB^3}{s^3(\log s)^3}$ with $C$ a universal constant. We do not detail the proof which is very similar to what we have done so far.

\myindent We then use lemma \eqref{commut}

\begin{equation}
    |-\scal{H^3 |y|^2 f_2}{f_1} + \scal{H^3 |y|^2 f_1}{f_2}| \leq C\|f\|_{H_x^3}^2
\end{equation}

\myindent Indeed we can write

\begin{equation}
    -\scal{H^3 |y|^2 f_2}{f_1} + \scal{H^3 |y|^2 f_1}{f_2} = \scal{H^{\frac{3}{2}} f_1}{\left[H^{\frac{3}{2}},|y|^2\right] f_2} + \scal{H^{\frac{3}{2}} f_2}{\left[H^{\frac{3}{2}},|y|^2\right] f_1}
\end{equation}

\myindent Thus, by choosing $B$ big enough with regards to universal constants, there holds

\begin{equation}
    |\partial_s \mathcal{E}(f^M)| \leq \frac{C B^3}{s^3 (\log s)^3}
\end{equation}

\myindent Hence

\begin{equation}
    \mathcal{E}(f^M(s)) \leq \frac{C B^3}{s^2 (\log s)^3}
\end{equation}

with $C$ a universal constant. Thus, by choosing $s_0$ big enough, there holds

\begin{equation}
    \mathcal{E}(f^M(s)) \leq \frac{1}{6} \frac{B^2}{s^2 (\log s)^2}
\end{equation}

\subsubsection{\textit{\textmd{Conclusion}}} 

\myindent We have proved that there exists a suitable choice of $B$ and $s_0$ such that while $T_{[s_0,M]}(B) \leq \sigma \leq M$ there holds

\begin{equation}
    \mathcal{E}(f^M(\sigma)) + E(f^M(\sigma)) + \scal{f_2^M(\sigma)}{\rho}^2 \leq \frac{1}{2}\frac{B^2}{\sigma^2 (\log \sigma)^2}
\end{equation}

\myindent Which ensures that $T_{[s_0,M]}(B) = s_0$ for all $M \geq s_0$, which is the announced result.

\myindent Now, lemma \eqref{boundf} and the bound of $r^M$ \eqref{normr} yield that $w^M$ is well-defined on $(s_0,M)$ for all $M> s_0$ and that there exists $B > 0$ such that 

\begin{equation}
    \forall s_0 < s < M \quad \|w^M(s)\|_{H_x^3} \leq \frac{B}{s \log(s)}
\end{equation}

\subsection{\textit{Cauchy sequence}}
\myindent Let us now show that $(w^M)$ is a Cauchy sequence in $\mathcal{C}([s_0,T], H_x^3)$ when $M \to \infty$ for all $T$. We use the notations introduced so far : let $s_0, B$ as built before, and for $s_0 < N < M$ let us see that 

\begin{equation}
    \partial_s(w^M - w^N) = \mathcal{L}(w^M - w^N) + IK(s)(w^M - w^N) + I(N^M(s) - N^N(s))
\end{equation}

\myindent This ensures that

\begin{equation}
    \begin{split}
    \partial_s \scal{(w_2^M - w_2^N)(s)}{\rho} &= -\scal{(w_1^M - w_1^N)(s)}{Q} - \beta(s)\scal{w_1^M - w_1^N}{|y|^2Q - \alpha Q^2}\\
    &+ \scal{N^M_1(s) - N^N_1(s)}{Q}\\
    \partial_s E(w^M - w^N) &= \scal{\mathcal{H}IK(s)(w^M - w^N)}{w^M - w^N}\\
    &+ \scal{\mathcal{H} (w^M - w^N)}{I(N^M(s) - N^N(s))}\\
    \partial_s \mathcal{E}(w^M - w^N) &= \scal{\mathcal{H}_3 IK(s)(w^M - w^N)}{w^M - w^N}\\
    &+ \scal{\mathcal{H}_3(w^M - w^N)}{I(N^M - N^N)}
    \end{split}
\end{equation}

\myindent Using \eqref{algebra} we thus have up to a bigger $s_0$ that 

\begin{equation}
    \begin{split}
    \|N^M(s) - N^N(s)\|_{H_x^3} &\leq C\|w^M(s) - w^N(s)\|_{H_x^3}(\|w^M(s)\|_{H_x^3} + \|w^N(s)\|_{H_x^3})\\
    &\leq \frac{CB}{s\log s}\|w^M(s) - w^N(s)\|_{H_x^3}
    \end{split}
\end{equation}

\myindent With the same techniques as before, there holds

\begin{equation}
    \begin{split}
    |\partial_s E(w^M - w^N)| &\leq \frac{C B}{s \log s} \|w^M - w^N\|_{H_x^3}^2 \\
    |\partial_s \mathcal{E}(w^M - w^N)| &\leq \frac{CB}{s \log s}\|w^M - w^N\|_{H_x^3}^2
    \end{split}
\end{equation}

\myindent Moreover using $\scal{w_1^M}{Q} = -\frac{1}{2}\|w^M\|_{L^2}^2$ we find

\begin{equation}
    \begin{split}
    \left| \scal{w_1^M(s) - w_1^N(s)}{Q}\right| &= \frac{1}{2}\left|\|w^M(s)\|_{L^2}^2  - \|w^N(s)\|_{L^2}^2\right|\\
    &\leq C\|w^M(s) - w^N(s)\|_{H_x^3}(\|w^M(s)\|_{H_x^3} + \|w^N(s)\|_{H_x^3})\\
    &\leq \frac{CB}{s \log s}\|w^M(s) - w^N(s)\|_{_x^3}
    \end{split}
\end{equation}

\myindent From which we infer

\begin{equation}
    \left|\partial_s \scal{w_2^M(s) - w_2^N(s)}{\rho}\right| \leq \frac{C B}{s \log s} \|w^M(s) - w^N(s)\|_{H_x^3}
\end{equation}

and thus

\begin{equation}
    \begin{split}
    \left|\partial_s\left(\scal{w_2^M(s) -w_2^N(s)}{\rho}^2\right)\right| &= 2 \left|\scal{w_2^M(s) -w_2^N(s)}{\rho}\right|\left| \partial_s \scal{w_2^M(s) -w_2^N(s)}{\rho}\right|\\
    &\leq \frac{CB}{s \log s} \|w^M - w^N\|_{H_x^3}^2
    \end{split}
\end{equation}

\myindent Integrating between $s$ and $N$, and using $w^N(N)= 0$, yields

\begin{equation}
    \begin{split}
    \scal{w_2^M(s) -w_2^N(s)}{\rho}^2 &\leq \scal{w_2^M(N)}{\rho}^2 + \int_s^N \frac{CB}{\sigma \log \sigma}\|w^M(\sigma) - w^N(\sigma)\|_{H_x^3}^2 d\sigma \\
   E(w^M(s) - w^N(s)) &\leq E(w^M(N)) + \int_s^N \frac{CB}{\sigma \log \sigma} \|w^M(\sigma) - w^N(\sigma)\|_{H_x^3}^2 d\sigma\\
   \mathcal{E}(w^M(s) - w^N(s)) &\leq \mathcal{E}(w^M(N)) + \int_s^N \frac{CB}{\sigma \log \sigma}\|w^M(\sigma) - w^N(\sigma)\|_{H_x^3}^2 d\sigma
   \end{split}
\end{equation}

\myindent Thus there exists a universal constant $C$ such that for all $s_0 \leq s \leq N$  

\begin{equation}
    \|w^M(s) - w^N(s)\|_{H_x^3}^2 \leq C\|w^M(N)\|_{H_x^3}^2 + \int_s^N \frac{C}{\sigma \log \sigma}\|w^M(\sigma) - w^N(\sigma)\|_{H_x^3}^2 d\sigma
\end{equation}

\myindent The Gronwall lemma, and the fact $\|w^M(N)\|_{H_x^3}^2 \leq \frac{C}{N^2 (\log N)^2}$ yield 

\begin{equation}
    \begin{split}
    \|w^M(s) - w^N(s)\|_{H_x^3}^2 &\leq \frac{C}{N^2 (\log N)^2} + \int_s^N \frac{C}{N^2 (\log N)^2 \sigma \log \sigma} \exp\left(\int_s^N \frac{C}{\sigma \log \sigma}\right)\\
    &\leq \frac{C}{N^2 (\log N)^2}\left(1 + \int_s^N \frac{(\log \sigma)^{C-1}}{\sigma} d\sigma\right)\\
    &\leq \frac{C'}{N}
    \end{split}
\end{equation}

where $C'$ is a constant.

\myindent Thus, $(w^M)$ is a Cauchy sequence in $\mathcal{C}((s_0,T),H_x^3)$ which converges uniformly to a $w(s)$. Using integral form of the equation, $w$ is a $\mathcal{C}^1$ solution to \eqref{eqw} on $(s_0, +\infty)$ with the bound

\begin{equation}
    \|w(s)\|_{H_x^3} \leq \frac{B}{s\log s}
\end{equation}

\myindent This concludes the proof of proposition \eqref{princprop}. Let us observe that the only role that the potential $W$ plays (hence $V$ in theorem \eqref{main}) is to cancel the contribution of the eigenvalue $\mu_1$ of $A$ in the oscillations as if $\mu_1 = 4$ then $\cos((\mu_1 - 4)\sigma)$ would be a constant and we wouldn't be able to obtain our estimates. We weren't able to explicitly compute the value of $\mu_1$ when $\eps \to 0$, but let us stress that if one could prove that there are arbitrarily small $\eps$'s for which $\mu_1 \neq 4$ then one could set $V = 0$ in theorem \eqref{main}. 

\section{Proof of Theorem \eqref{main}}

\myindent Let us choose $s_0$ such as built in the previous section, $t(s_0) = s_0$ and $\frac{dt}{ds} = L^2(s)$ as in proposition \eqref{restraj}. We have built

\begin{equation}
    u(t,x) = e^{-i\lambda s(t)} S_{L(t)}\left( e^{-i\frac{b(t)|\cdot|^2}{4}} v(s(t),\cdot) \right)(x)
\end{equation}

a solution to

\begin{equation}
    i\partial_t u(t,x) = Hu(t,x) + u(t,x)|u(t,x)|^2 + V(t,x) u(t,x)
\end{equation}

with the potential

\begin{equation}
    V(t,x) = -\alpha\frac{1}{L^2(t)} \beta(s(t)) Q\left(\frac{x}{L(t)}\right)
\end{equation}

\myindent Using proposition \eqref{restraj} there holds for all $k,r$ that 

\begin{equation}
    \lim_{t\to\infty} \|\partial_t^k V(t,x)\|_{H_x^r} = 0
\end{equation}

\myindent Now, we remind that 

\begin{equation}
    v(s,y) = Q(y) + w(s,y)
\end{equation}

\myindent Thus we can decompose

\begin{equation}
    u(t,x) = u_0(t,x) + u_1(t,x)
\end{equation}

where

\begin{equation}
    u_0(t,x) = e^{-i\lambda s(t)} S_{L(t)} \left(e^{-\frac{ib(t)|\cdot|^2}{4}} Q\right)(x)
\end{equation}

and

\begin{equation}
    u_1(t,x) = e^{-i\lambda s(t)} S_{L(t)}\left( e^{-\frac{ib(t)|\cdot|^2}{4}} w(s,\cdot)\right)(x)
\end{equation}

\myindent Thus, propositions \eqref{princprop} and \eqref{restraj} along with lemma \eqref{enest} yield that

\begin{equation}
\begin{split}
\|u_1(t,x)\|_{H_x^1}^2 &= \left\|S_{L(t)}\left( e^{-\frac{ib(t)|\cdot|^2}{4}} w(s,\cdot)\right)(x)\right\|_{H_x^1}^2\\
 &\leq 2E(L,b) \|w(s,y)\|_{H_x^1}^2 \\
    &\leq C'\log(t) \frac{1}{t^2(\log t)^2}
    \end{split}
\end{equation}

\myindent Therefore 

\begin{equation}
    \|u_1(t,x)\|_{H_x^1} = O\left(\frac{1}{t \left(\log(t)\right)^{\frac{1}{2}}}\right) << \left(\log(t)\right)^{\frac{1}{2}}
\end{equation}

\myindent Similarly, lemma \eqref{enest} and proposition \eqref{restraj} yield

\begin{equation}
    \begin{split}
    \|u_0\|_{H_x^1}^2 &= \left(L^2 + \frac{b^2}{4L^2}\right)\|xQ\|_{L^2}^2 + \frac{1}{L^2}\|\nabla Q\|_{L^2}^2 \\
    &\geq E(L,b) \min(\|xQ\|_{L^2}^2, \|\nabla Q\|_{L^2}^2)\\
    &\geq C \log(t)
    \end{split}
\end{equation}

and

\begin{equation}
    \begin{split}
    \|u_0\|_{H_x^1}^2 &= \left(L^2 + \frac{b^2}{4L^2}\right)\|xQ\|_{L^2}^2 + \frac{1}{L^2}\|\nabla Q\|_{L^2}^2 \\
    &\leq E(L,b) \max(\|xQ\|_{L^2}^2, \|\nabla Q\|_{L^2}^2)\\
    &\leq C' \log(t)
    \end{split}
\end{equation}

\myindent Thus we have proven that there exists constants $c, C > 0$ such that 

\begin{equation}
    c\left(\log(t)\right)^{\frac{1}{2}} \leq \|u_0(t)\|_{H_x^1} \leq C\left(\log(t)\right)^{\frac{1}{2}}
\end{equation}

and we have similar bounds for $u$ itself as $\|u_1(t,x)\|_{H_x^1} \to 0$.

\myindent This nearly concludes the proof to theorem \eqref{main}. It remains to be shown that $u$ has $\mathcal{C}^{\infty}$ regularity. However that is a straightforward consequence of local uniqueness and the fact that $Q,L,b$ have $\mathcal{C}^{\infty}$ regularity.

\printbibliography
\end{document}